\documentclass[a4paper,12pt,twoside]{amsart}

\usepackage{amsfonts}
\usepackage{amsfonts,amssymb,amscd}
\usepackage{graphicx}
 \usepackage[left=2.3cm,top=3.3cm,right=2.3cm,bottom=2.7cm]{geometry}
\usepackage{mathrsfs}
\usepackage{xfrac}
\let\mathcal\mathscr
\usepackage[colorlinks=true,citecolor=blue, urlcolor=blue, breaklinks, linkcolor=black]{hyperref}

\usepackage[all,ps,cmtip]{xy}

\DeclareRobustCommand{\SkipTocEntry}[5]{}

\def\R{{\bf R}}

\def\llra{\hbox to 10mm{\rightarrowfill}}

\def\lllra{\hbox to 15mm{\rightarrowfill}}

\def\PA{{\widehat A}}
\def\PB{{\widehat B}}
\def\PK{{\widehat K}}

\def\PT{{\widehat T}}

\def\phi{{\varphi}}
\def\wf{{\widetilde f}}
\def\wX{{\widetilde X}}

\def\wZ{{\widetilde Z}}
\def\wY{{\widetilde Y}}

\def\cI{\mathcal{I}}

\def\cF{\mathcal{F}}

\def\cO{\mathcal{O}}

\def\cP{\mathcal{P}}

\def\cM{\mathcal{M}}

\def\cQ{\mathcal{Q}}

\def\cZ{\mathcal{Z}}

\let\tilde\widetilde

\DeclareMathOperator{\rank}{rank}

\DeclareMathOperator{\codim}{codim}
\DeclareMathOperator{\Pic}{Pic}

\newtheorem{lemm}{Lemma}[section]
\newtheorem{theo}[lemm]{Theorem}
\newtheorem{coro}[lemm]{Corollary}
\newtheorem{prop}[lemm]{Proposition}

\theoremstyle{definition}
\newtheorem{defi}[lemm]{Definition}
\newtheorem{rema}[lemm]{Remark}

\newtheorem{exam}[lemm]{Example}

\newtheorem{qu}[lemm]{Question}

\theoremstyle{remark}
\newtheorem*{remark*}{Remark}
\newtheorem*{note*}{Note}

\begin{document}
\title{Irregular varieites with geometric genus one, theta divisors, and fake tori}
\author{Jungkai Chen}
\address{Department
of Mathematics, National Taiwan University, Taipei 106, Taiwan, and National Center for Theoretical Sciences\\ 1 Sec. 4, Roosevelt Rd. Taipei 106, Taiwan}
\email{jkchen@ntu.edu.tw}
\author[Z.~Jiang]{Zhi Jiang}
\address{D\'epartement de Math\'ematiques d'Orsay\\UMR CNRS 8628\\Universit\'{e} Paris-Sud\\B\^{a}timent 425, F-91405 Orsay, France}
\email{Zhi.Jiang@math.u-psud.fr}

\author[Z.~Tian]{Zhiyu Tian}
\address{CNRS\\Institut Fourier, UMR 5582\\Universit\'e Grenoble Alpes CS 40700\\ 38058, Grenoble, France}
\email{zhiyu.tian@ujf-grenoble.fr}

\begin{abstract}
We study the Albanese image of a compact K\"ahler manifold whose geometric genus is one.
In particular, we prove that if the Albanese map is not surjective, then the manifold maps surjectively onto an ample divisor in some abelian variety, and in many cases the ample divisor is a theta divisor.
With a further natural assumption on the topology of the manifold, we prove that the manifold is an algebraic fiber space over a genus two curve.
Finally we apply these results to study the geometry of a compact K\"ahler manifold which has the same Hodge numbers as those of an abelian variety of the same dimension.
\end{abstract}
\maketitle
\section{Introduction }
Kawamata proved in \cite{kaw} that if $X$ is a smooth projective variety with Kodaira dimension $\kappa(X)=0$, then the Albanese morphism $a_X: X\rightarrow A_X$ is an algebraic fiber space.  An effective version of this result was obtained in \cite{J}. For instance, the author proved that if $p_g(X)=P_2(X)=1$,  $a_X$ is an algebraic fiber space.  Pareschi, Popa and Schnell recently prove the same criterion for compact K\"ahler manifolds in \cite{PPS}.

On the other hand, if we only assume that $p_g(X)=1$, $a_X$ is not necessarily surjective.
In this article  we will show that, if  $p_g(X)=1$ and $a_X$ is not surjective,
the Albanese image  is closely related to the geometry of theta divisors.
\begin{theo}\label{fibration-to-divisor}
Let $X$ be a compact K\"ahler manifold with $p_g(X)=1$. Then
\begin{itemize}
\item[(1)] $\dim a_X(X)\geq \frac{1}{2}\dim A_X$;
\item[(2)] $a_X$ is not surjective if and only if there exists an ample divisor $D$ of an abelian variety $B$ with a surjective morphism $f: X\rightarrow D$.
 \end{itemize}
\end{theo}

 In a special case when $A_X$ is simple, we have:
\begin{theo}\label{image-theta}
 Let $X$ be a compact K\"ahler manifold with $p_g(X)=1$.
 Assume that $a_X$ is not surjective and  $A_X$ is simple. Then $a_X(X):=D$  is an ample divisor of $A_X$. Moreover, if $D$ is smooth in codimension $1$,  then $D$ is a theta divisor of $A_X$ and $a_X$ is a  fibration onto $D$.
\end{theo}

One ingredient of the proof of Theorem \ref{fibration-to-divisor}  is the decomposition theorem of Pareschi-Popa-Schnell in \cite{PPS}, which establishes generic vanishing theory for Hodge modules on compact K\"ahler manifolds. By the decomposition theorem,  the ``positive" part of $a_{X*}\omega_X$ comes from algebraic varieties and this allows us to reduce the statement to the algebraic setting.

Part $(1)$ is a generalization of the main theorem of   \cite{hp}.
The ``if" part of $(2)$ is clear. If there exists a surjective morphism from $X\rightarrow D $,
the induced morphism $g: X\rightarrow D\hookrightarrow B$ factors through $a_X: X\rightarrow A_X$.
Then $a_X$ is not surjective. The ``only if" part  can  then be proved using the  idea in Pareschi's
characterization of theta divisors (see \cite{P}). In Section 3 and 4, we will see much more precise
structures of $a_X(X)$ and why it should be related to theta divisors.

With a further assumption on the second Betti cohomology, we have a very strong conclusion:
 \begin{theo}\label{fibration-to-curve}
 Let $X$ be a compact K\"ahler manifold with $p_g(X)=1$. Then the pull-back map $a_X^*: H^2(A_X, \mathbb{Q})\rightarrow H^2(X, \mathbb{Q})$
 is not injective if and only if there exists a fibration $\varphi: X\rightarrow C$ to a  smooth projective curve $C$ of genus $2$.
 \end{theo}
 The ``if" part is again clear. The fibration $f$ induces a fibration $A_X\rightarrow JC$. Since $C$ is a curve, $H^2(JC, \mathbb{Q})\rightarrow H^2(C, \mathbb{Q})$ is not injective. Hence $H^2(A_X, \mathbb{Q})\rightarrow H^2(X, \mathbb{Q})$ is also not injective.

A more careful analysis shows the following.

\begin{coro}\label{cor:deRham}
 Let $X$ be a compact K\"ahler manifold with $p_g(X)=1$.
Then the de Rham fundamental group $\pi_1(X) \otimes \mathbb{Q}$ of $X$ is isomorphic to a product of $\mathbb{Q}^{2r} \times (\pi_1(C)\otimes \mathbb{Q})^i$,
where $C$  is a smooth curve of  genus $2$ and $i=\frac{1}{5}\dim (\text{Ker}(H^2(A_X, \mathbb{Q})=\Lambda^2 H^1(X, \mathbb{Q})\rightarrow H^2(X, \mathbb{Q}))), 2r+2i=b_1(X)$ and $i$ is less or equal to the codimension of the the Albanese image of $X$.
\end{coro}

The de Rham fundamental group is the $\mathbb{Q}$-unipotent completion of the topological fundamental group (see \cite{KahlerGroup}). One can construct examples where $i$ takes all possible values $0, 1, \ldots, s=\codim_{A_X} a_X(X)$.

The  motivation  to study the Albanese image of irregular varieties with geometric genus one comes from an explicit geometry question.

Catanese showed that a  compact K\"ahler manifold, whose integral cohomology ring is isomorphic to that of a torus, is actually a complex torus.

In \cite{DJL}, the authors study projective varieties $X$ with mild singularities, whose rational cohomology rings are isomorphic to those of complex tori.  These varieties are called rational cohomology tori.  The Albanese morphism of a rational cohomology torus is  finite  and  is often an abelian cover of the Albanese variety.

It is then natural to ask  what can we say about the general structure of $X$ if we further loosen the condition on cohomology rings.
The  condition that $\dim H^i(X, \mathbb{Q})=\dim H^i(A_X, \mathbb{Q})$  is  too weak to say anything interesting. Indeed, by blowing-up subvarieties on $\mathbb{P}^m$-bundles over curves, we can construct many varieties verifying this condition.

 On the other hand, Betti cohomology of smooth projective varieties carries Hodge structures, which  usually inherit information about the complex strucutre of $X$.
 John Ottem  asked:
\begin{qu}
Let $X$ be a compact K\"ahler manifold. Assume that $h^{p, q}(X)=h^{p,q}(A_X)$ for all $p$ and $q$. Is $X$ a rational cohomology torus ?
\end{qu}

Note that the above question is equivalent to ask whether the Albanese morphism $a_X$ is generically finite under the assumption of Hodge numbers. If so,  the pull-back $a_X^*: H^*(A_X, \mathbb{Q})\rightarrow H^*(X, \mathbb{Q})$ would be an isomorphism and hence $X$ is a rational cohomology torus.

The answer to  Ottem's question is negative. A counter-example is described in \cite[Example 1.7]{DJL},
which is an elliptic curve fibration over a genus $2$ curve. We will see that,  despite the counterexamples, there are
strong  restrictions on  the structure of $a_X: X\rightarrow A_X$.

 \begin{defi}
Let $X$ be a compact K\"ahler manifold. We say that $X$ is a fake torus if the Hodge numbers of $X$ are the same as those of a complex torus of the same dimension and the Albanese morphism $a_X$ is not generically finite.
\end{defi}

 The following result is a direct application of Theorem \ref{fibration-to-curve}.

 \begin{coro}
Let $X$ be a fake torus. There exists a fibration $f: X\rightarrow C$ to a smooth projective curve $C$ of genus $2$.
In particular, the fundamental group of a fake torus is non-abelian.
\end{coro}
\begin{proof}
By definition of fake torus, $p_g(X)=1$ and $a_X$ is not surjective. Hence  any K\"ahler class on $X$ does not come from $A_X$. Moreover, since $\dim H^2(A_X, \mathbb{Q})=\dim H^2(X, \mathbb{Q})$, the pull-back $a_X^*: H^2(A_X, \mathbb{C})\rightarrow H^2(X, \mathbb{C})$ is not injective. Then Theorem \ref{fibration-to-curve}  implies that there exists a fibration $f: X\rightarrow C$. Then we have a surjective map $\pi_1(X)\twoheadrightarrow \pi_1(C)$. Hence $\pi_1(X)$ is not abelian.
\end{proof}

We have a good understanding on the structure of $a_X(X)$ for a fake torus  $X$, thanks to the theory of generic vanishing.  However, the fiber of $a_X$ is poorly understood. That is the reason that we don't have a picture of the general structure of fake tori. Nevertheless, when $\dim a_X(X)=\dim X-1$, we have the following general result.
\begin{theo}Let $X$ be a fake torus of dimension $n$. If $\dim a_X(X)=n-1$, then $X$ is not of general type.
\end{theo}

Moreover, we can also describe explicitly fake tori in low dimensions.

\subsection*{Acknowledgements}
We thank John Ottem for his question and Stefen Schreieder for discussions. The second author worked on this article during his visits at Shanghai Center for Mathematical Sciences,  National Center for Theoretical Sciences (NCTS) of Taiwan, and Yau Mathematical Sciences Center, and he is grateful for all the supports from these institutes.
\section{Notations and preliminaries}\label{notation}
\subsection{Subvarieties of general type}
A  compact K\"ahler manifold is always connected and a variety is always supposed to be reduced and irreducible.

 A subvariety of a torus is called of general type if any of its desingularization is a smooth projective variety
 of general type. Ueno (\cite[Theorem 10.9]{U}) proved that a subvariety of a complex torus is not of general type
 if and only if it is fibred by sub-torus.
 More precisely, given a subvariety $Z$ of a complex torus $B$.  Let $K$ be the  maximal subtorus of $B$ such that $K+Z=Z$ and denoted $B^\flat=Z/K$. Then there is $Z^\flat \subset B^\flat$ such that $Z^\flat$ is of general type and $Z \to Z^\flat$ is fibred by $K$. We call $Z^\flat$ (resp. $Z^\flat \subset B^\flat$) the {\it $\kappa$-reduction of $Z$} (resp. of $Z \subset B$). We call $K$ the $\kappa$-kernel of $Z$.
 Notice that if $Z$ is of general type, then clearly $Z^\flat=Z$. Hence one has $(Z^\flat)^\flat = Z^\flat$ in general.

  Let $X$ be a compact K\"ahler manifold. We denote by $Y \subset A_X$ the image of the Albanese morphism of $X$. In sequel, we will fix the following notation:

\begin{eqnarray}\label{setting}
\xymatrix{X\ar@/_2pc/[dd]_f\ar[dr]^{a_X}\ar[d]_g\\
Y\ar[d]_h\ar@{^{(}->}[r] & A_X\ar[d]^p\\
Z\ar@{^{(}->}[r]  &B,}
\end{eqnarray}
where $Z \subset B$ is the $\kappa$-reduction of $Y \subset A_X$.
 In this setting, if $Y$ is of general type, then $h$ and $p$ are respectively isomorphisms of $Y$ and $A_X$.

 For any torus $A$, we will denote by $\PA=\Pic^0(A)$ the dual abelian variety.

We adapt the following notations.  Assume that $Z\hookrightarrow A$ is a subvariety (possibly of general type) of an abelian variety and assume that $\PB$ is an abelian subvariety of $\PA$. Let $B$ be a quotient abelian variety of $A$ and $Z_B$ denote the image of $Z$ in $B$. The $\kappa$-reduction of $Z_B$, denoted $Z_B^\flat$, is called the {\it $\kappa$-reduction of $Z$ with respect to $B$}. We will need the following easy Lemma.

\begin{lemm}\label{easyprop}
Let $Z \hookrightarrow A$ be a subvariety of an abelian variety, $\PB_2\subset \PB_1\subset \PA$ be abelian subvarieties, and $Z_{B_i} \hookrightarrow B_i$ be the  $\kappa$-reduction with respect to $B_i$. Then there is induced commutative diagram with surjective vertical morphisms
\begin{eqnarray}
\xymatrix{Z_{B_1}^\flat \ar[d] \ar@{^{(}->}[r] & B_1^\flat \ar[d]\\
Z_{B_2}^\flat \ar@{^{(}->}[r]  & B_2^\flat.}
\end{eqnarray}
In particular, the torus $\widehat{B_2^\flat}$ is a subtorus of $\widehat{B_1^\flat}$.
\end{lemm}


\begin{lemm}\label{easy}
Let $Z\hookrightarrow B$ be a subvariety of general type. For $\PB_1$ and $\PB_2$ two abelian subvarieties of $\PB$. Let $\widehat{B_{12}}$ be the neutral component of $\PB_1\cap \PB_2$ and let $\widehat{B^1_2}$ be the neutral component of $\widehat{B_1^\flat} \cap  {\PB_2}$. Then $(B^1_2)^\flat = (B_{12})^\flat$.
\end{lemm}
\begin{proof}
It is clear that $\widehat{B^1_2} \subset  \widehat{B_{12}}$ and hence $\widehat{(B^1_2)^\flat} \subset  \widehat{(B_{12})^\flat}$ by Lemma \ref{easyprop}.
Moreover, since $\widehat{(B_{12})^\flat} \subset \widehat{B_{1}^\flat}$ and $\widehat{(B_{12})^\flat} \subset \widehat{B_{2}^\flat} \subset \widehat{B_2}$, one has
$\widehat{(B_{12})^\flat} \subset \widehat{B^1_2}$. Hence
$$ \widehat{(B_{12})^\flat} = \widehat{{(B_{12})^\flat}^\flat} \subset \widehat{(B^1_2)^\flat}.$$ This completes the proof.
\end{proof}

\subsection{Hodge type sheaves}
 In this article, we often consider a torsion-free coherent sheaf $\cF$ on a subvariety $i: Z\hookrightarrow A$  (we usually do not distinguish $\cF$ and $i_*\cF$)
 satisfying the following properties:
 \begin{itemize}
  \item[(P1)] $\cF$ is a GV sheaf on $A$;
  \item[(P2)] for all $i, k\geq 0$, the cohomological support loci $$V^i(\cF):=\{P\in\Pic^0(A)\mid \dim H^i(\cF\otimes P)>0\}$$ and
  $$V^i_k(\cF):=\{P\in\Pic^0(A)\mid \dim H^i(\cF\otimes P)\geq k\}$$ are  union of torsion translated abelian subvarieties of $\Pic^0(A)$;
  \item[(P3)] let $g: A\rightarrow B$ be a morphism between abelian varieties, let $Z_B$ be the image of $Z$, and let $r=\dim Z-\dim Z_B$, then $$\mathbf{R}g_{*}(\cF\otimes Q)=\bigoplus_{0\leq j\leq r}R^jg_*(\cF\otimes Q)[-j] \in D^b(B),$$ for any torsion line bundle $Q\in \Pic^0(A)$;
  \item[(P4)]   moreover,  $R^jg_*(\cF\otimes Q)$ is either $0$ or is a torsion-free GV sheaf on $Z_B$.

   \end{itemize}

 We call a torsion-free coherent sheaf on $Z$ satisfying Properties $(1)$, $(2)$, $(3)$ and $(4)$ a {\it Hodge sheaf on $A$ supported on $Z$}.

 \begin{lemm}\label{basic}
Let $\cF$ be a Hodge sheaf on $A$ supported on $Z$.
\begin{itemize}
\item[(1)] If for some $i>0$, $V^i(\cF)$ has a component $P_0+\PB_0$ of codimension $i$,  where $P_0$ is a torsion line bundle and $\PB_0$ is an abelian subvariety. Let $K$ be the kernel of $A\rightarrow B_0$, then $Z+K=Z$.  In particular, if $Z$ is of general type, $\cF$ is M-regular.
\item[(2)] Let $g: A\rightarrow B$ be a morphism between abelian varieties and let $Q\in \Pic^0(A)$ a torsion line bundle. If $R^jg_*(\cF\otimes Q)\neq 0$ for some $j\geq 0$, $R^jg_*(\cF\otimes Q)$ is a Hodge sheaf on $B$ supported on $Z_B$.
\item[(3)] Let $\cF'$ be a direct summand of $\cF$, then $\cF'$ is also a Hodge sheaf.
\end{itemize}
\end{lemm}
\begin{proof}
The proof of $(1)$ is standard, see for instance \cite[Lemma 1.1]{JLT}. $(3) $ is also clear.

For $(2)$, let $\cQ:=R^jg_*(\cF\otimes Q)$.
Then $\cQ$ is a GV sheaf on $Z_B$ by $(P4)$. By $(P3)$, we know that, for any $j\geq 0$ and $Q'\in \Pic^0(B)$,
$$h^j(A, \cF\otimes Q\otimes g^*Q')=\sum_{s+t=j}h^s(B, R^tg_*(\cF\otimes Q)\otimes Q').$$
Since all $V^j_k(\cF)$ are unions of torsion translated abelian subvarieties of $\Pic^0(A)$, all $V^s_k(\cQ)$  are union of torsion translated abelian subvarieties of $\Pic^0(B)$. Let $g': B\rightarrow B'$ be a morphism between abelian varieties and let $f=g'\circ g$. Then
$$\R f_*(\cF\otimes Q)=\R g'_*\R g_*(\cF\otimes Q)=\R g'_*(\bigoplus_j R^jg_*\big(\cF\otimes Q)[-j] \big).$$ We then conclude that $\cQ$ satisfies $(P3)$ and $(P4)$ by the same argument as in \cite[Theorem 3.4]{K2}.
\end{proof}

We call a Hodge sheaf on $A$ supported on $Z$ a {\it strong Hodge sheaf } if it satisfies furthermore the following two properties.
 \begin{itemize}
 \item[(P5)] For any morphism $g: A\rightarrow B$ between abelian varieties and for any $Q\in \Pic^0(A)$,
let $\epsilon: Z_B'\rightarrow Z_B$ be a desingularization,  then there exists a torsion-free coherent sheaf $\cF_Q$ on $Z_B'$ such that
  $\mathbf{R}\epsilon_{*}\cF_Q=g_*(\cF\otimes Q)$ and $\cF_Q\otimes \omega_{Z_B'}^{-1}$ is weakly positive on $Z_B'$;
  \item[(P6)] Let $g: A\rightarrow B$ be as above. For $b\in Z_B$ general, denote  $j: Z_b\hookrightarrow K$ the fibers of $Z$ and $A$ over $b$. Then $\cF\mid_{Z_b}$ satisfies $(P1)-(P5)$.
   \end{itemize}
 \begin{rema}
 It is clear that $\cF\mid_{Z_b}$ is a strong Hodge sheaf on $K$ supported on $Z_b$.
 \end{rema}
 \begin{lemm}\label{summarize}
Let $f: X\rightarrow A$ be a morphism from a compact K\"ahler manifold to an abelian variety. Let $\cF=R^jf_*\omega_X$ for some $j\geq 0$. Then $\cF$
is a strong Hodge sheaf on $A$ supported on $f(X)$.
\end{lemm}

\begin{proof}
First of all, \cite[Theorem A]{PPS} implies the property $(P1)$ for $\cF$.

The property $(P3)$ was proved by Koll\'ar in \cite{K2} when $X$
is projective and was proved by Saito in general (see \cite{S1} and \cite{S2}, or \cite[Theorem 14.2]{PPS}).

Combining $(P2)$ with the work of Green-Lazarsfeld \cite{GL}, we know that all cohomological support loci $V^i(\cF)$ are
translated abelian subvarieties of $\Pic^0(A)$. The fact that $V^i(\cF)$ always contains a torsion point is first proved by Wang in \cite{W},
see also \cite[Corollary 17.1]{PPS}.

Let $Z\rightarrow X$ be the \'etale cover induced by the torsion line bundle $Q$ and let $h: Z\rightarrow B$ be the composition
of morphisms. Then $R^kg_*(\cF\otimes Q)$ is a direct summand of $R^{j+k}h_*\omega_{Z}$  and hence we have $(P4)$.

For $(P5)$, we  consider a birational modification $\pi: X'\rightarrow X$ between compact K\"ahler manifolds such that the composition of
morphisms $f\circ \pi: X'\rightarrow Z_B$ factors through $\epsilon$ as follows: $X'\xrightarrow{f'} Z\xrightarrow{g'} Z_B'\xrightarrow{\epsilon} Z_B$.
Then by Saito's decomposition, we have $g_*(R^jf_*\omega_X\otimes Q)=\epsilon_*(g'_*R^jf'_*\omega_{X'})$. Note that  $g'_*R^jf'_*\omega_{X'}$ is a direct summand of $R^j(g'\circ f')_*\omega_{X'}$ by Saito's decomposition. Moreover,  $R^j(g'\circ f')_*\omega_{X'/Z_B'}$ is weakly positive (see Popa \cite[Theorem 10.4]{Pop} or Schnell \cite[Theorem 1.4]{Sch}). Hence $g'_*R^jf'_*\omega_{X'/Z_B'}$ is also weakly positive.

Finally, by base change, $\cF\mid_{Z_b}=R^jf_{b*}\omega_{X_b}$, where $f_b: X_b\rightarrow K$ is the induced morphism between fibers. Hence we also have $(P6)$.
\end{proof}

 \begin{rema}
More generally, the properties $(P1)-(P6)$ are satisfied by certain coherent sheaves which are graded pieces of the underlying $D$-modules of mixed Hodge modules.

Let $M=(\mathcal{M}, F_{\cdot}\mathcal{M}, M_{\mathbb{Q}})$ be a polarizable Hodge module on an abelian variety $A$. Then for each $k\in \mathbb{Z}$, the coherent sheaf $gr^F_k\cM$ satisfies $(P1)$ and $(P2)$ (see \cite{PPS}). Moreover, let $p$ be the smallest number such that $F_p\cM\neq 0$ and let $S(\cM):=F_{p}\cM$. Then Saito (\cite{S2}) showed that $S(\cM)$ satisfies $(P3)$ and $(P4)$. By Popa and Schnell's result on weakly positive properties of $S(\cM)$, it is also easy to show that $S(\cM)$ satisfies $(P5)$ and $(P6)$.

 \end{rema}

\section{Hodge sheaf $\cF$ supported on $Z$ with $\chi(\cF)=1$}
In this section,  we will prove Theorem \ref{fibration-to-divisor}.
We first prove a general but technical result  on the structure of $Z$ and Theorem \ref{fibration-to-divisor} is a direct consequence.

The following lemma is essentially due to Pareschi (see \cite{P}).
\begin{lemm}\label{pareschi}
Let $Z\hookrightarrow B$ be a subvariety of general type.
Let $\cF$ be a torsion-free sheaf on $Z$ such that $\cF$ is M-regular on $B$ and $\chi(B, \cF)=1$.
Then $V^1(\cF)\neq \emptyset$.
Moreover, for a component $W$ of $V^1(\cF)$, if $\codim_{\PB}W=j+1\geq 2$, then $W$ is indeed a component of $V^j(\cF)$.
\end{lemm}

\begin{proof}
Indeed, we can apply the argument of Pareschi in the proof
of \cite[Theorem 5.1]{P}.
Denote by $\R\Phi: \mathrm{D}(B)\rightarrow \mathrm{D}(\PB)$ and $\R\Psi: \mathrm{D}(\PB)\rightarrow \mathrm{D}(B)$
the Fourier-Mukai functors induced by the normalized Poincar\'e line bundles $\cP$ on $B\times \PB$.
Let $\R\Delta(\cF):=\R\mathcal{H}om(\cF, \cO_B)\in \mathrm{D}(B)$.
Then Pareschi and Popa (\cite[Corollary 3.2]{PP1}) proved that $\R\Phi(\R\Delta(\cF))[-g]=\R^g\Phi(\R\Delta(\cF))$  is a torsion-free coherent sheaf on $\PB$ of rank  equal to $\chi(B, \cF)$, i.e. $1$.
Hence, we can write $\R\Phi(\R\Delta(\cF))[-g]=L\otimes \cI_{\cZ}$,
where $L$ is a line bundle on $\PB$ and $\cI_{\cZ}\hookrightarrow \cO_{\PB}$ is an ideal sheaf of
a subscheme $\cZ$  of of $\PB$.
On the other hand, by Mukai's formula,
we know that $(-1)^*\R\Psi(L\otimes \cI_{\cZ})=\R\Delta(\cF)$, whose support is also $Z$. Hence $\cI_{\cZ}$ is a proper subsheaf of $\cO_{\PB}$ otherwise the support of $(-1)^*\R\Psi(L\otimes \cI_{\cZ})$ is a union of translated abelian subvarieties.

Moreover, we know that
$(-1_{\PB})^*\R^{j}\Phi(\cF)\simeq \mathcal{E}xt^j(L\otimes \cI_{\cZ}, \cO_{\PB})$ (see for instance \cite[Proposition 1.6]{P}).
For any $j\geq 1$, $\mathrm{Supp}\big( \mathcal{E}xt^j(L\otimes \cI_{\cZ}, \cO_{\PB})\big)\subset \cZ$.
Hence for any $j\geq 1$, $\mathrm{Supp}(\R^{j}\Phi(\cF))\subset 0_{\PB}-\cZ$.
By cohomology and base-change, $V^1(\cF)\subset 0_{\PB}-\cZ$.

Let $W$ be a component of $V^1(\cF)$  of codimension $j+1\geq 2$. Then the support of
$\mathcal{E}xt^j(L\otimes \cI_{\cZ}, \cO_{\PB})\simeq \mathcal{E}xt^{j+1}(L\otimes \cO_{\cZ}, \cO_{\PB})$ has $0_{\PB}-W$ as a component.
Hence $V^j(\cF)$ has a component $W$ of codimension $j+1, j>0$.
\end{proof}

Let $Z$ be a subvariety of general type of an abelian variety $B$ of codimension $s$.
If $s=1$, $Z$ is an ample divisor of $B$.
The following lemma deals with the higher codimension cases.

\begin{lemm}\label{fibration}
 Let $Z\hookrightarrow B$ be a subvariety of general type
of codimension $s\geq 1$. Assume that there exists a Hodge sheaf $\cF$ supported on $Z$ with $\chi(Z, \cF)=1$.
For any component $T_i=P_i+\PB_i$ of $V^1(\cF)$, where $P_i$ is a torsion line bundle and $\PB_i$ is an abelian subvariety of $\PB$, let $Z_i:=Z_{B_i} $ and $Z_i^\flat:=Z_{B_i}^\flat$ be the $\kappa$-reduction of $Z_i$.  Then $\codim_{B_i}Z_i=\codim_{B_i^\flat} Z_i^\flat=s-1$.
\end{lemm}
\begin{proof}

Since $\cF$ is a Hodge sheaf, by definition, we know that each component $T_i$ of $V^1(\cF)$  can be written as $P_i+\PB_i$, where $P_i$ is a torsion line bundle and $\PB_i$ is an abelian subvariety of $\PB$,

We  consider the commutative diagram
\begin{eqnarray}\label{diagram-fibration}
\xymatrix{
Z\ar[d]_{h_i}\ar@/_2pc/[dd]_{h_i^\flat}\ar@{^{(}->}[r] & B\ar[d]^{p_i}\ar@/^2pc/[dd]^{p_i^\flat}\\
Z_i\ar[d]\ar@{^{(}->}[r]  &B_i\ar[d] \\
Z_{i}^\flat\ar@{^{(}->}[r]  &{B_i}^\flat,}
\end{eqnarray}
where  both the fibers of $Z_i\rightarrow Z_i^\flat$ and $B_i \rightarrow B_i^\flat$ are translates of abelian variety $R_{B_i}=\ker(B_i \to B_i^\flat)$.

By construction, $\codim_{B_i}Z_i=\codim_{B_i^\flat}Z_i^\flat$, we just need to show that $\codim_{B_i}Z_i=s-1$.

Assume that $\codim_{\PB}\PB_i=j+1$. Then, by Lemma \ref{pareschi}, $P_i+\PB_i$ is a component of $V^j(\cF)$,
by $(P3)$, we know that for any $Q\in \PB_i$, $$0\neq H^j(Z, \cF\otimes P_i\otimes Q)=\sum_{l+k=j}H^l(Z_i, R^kh_{i*}(\cF\otimes P_i)\otimes Q). $$
By  $(P4)$, all $R^kh_{i*}(\cF\otimes P_i)$ are GV-sheaves on $Z_i$. Thus for a general $Q$ the right hand side has a single term $H^0(Z_i, R^jh_{i*}(\cF\otimes P_i)\otimes Q)$.
We then conclude that $R^jh_{i*}(\cF\otimes P_i)$ is a non-trivial torsion-free sheaf on $Z_i$.
By  $(P3)$,  a general fiber of $h_i$ has dimension at least $j$. Moreover, $Z$ is of general type, thus a general fiber $Z_t$ of $h_i$ must be a divisor of $B_t$.   Hence $\dim Z_i=\dim Z-j$, $\dim B_i=\dim B-j-1$, and $\codim_{B_i}Z_i=\codim_BZ-1=s-1$.
\end{proof}


For the inductive purpose, we need the following more refined statement.

\begin{lemm}\label{correspondence} Keep the assumptions of Lemma \ref{fibration}. We then consider the diagram (\ref{diagram-fibration}) with $i=1$.
Let $Q\in \PB$ be a general torsion point (in particular, $Q\notin V^1(\cF)$) and consider
the sheaf $\cF_Q:=h_{1*}^\flat (\cF\otimes Q)$.

Then the map
\begin{eqnarray*}
&& \PB_k \mapsto \widehat{B^1_k} \colon=\mathrm{the\; neutral\; component\; of\;} \PB_k\cap \widehat{B_1^\flat}
\end{eqnarray*}
induces an bijection between the following sets:
$$S_1:=\{\PB_k\mid T_k=P_k+\PB_k\;\mathrm{is\; a\; component\; of}\; V^1(\cF)\; \mathrm{and}\; q(\PB_k)=\PB/\widehat{B_1^\flat}\},$$ and
$$S_2:=\{\PB'_k\mid T_k'=P_k'+\PB_k'\;\mathrm{is\; a\; component\; of}\; V^1(\cF_Q)\},$$
where $q: \PB \to \PB/\widehat{B_1^\flat}$  is the natural quotient.
\end{lemm}

\begin{proof}
 We have $h^0(Z^\flat_1, \cF_Q)=h^0(Z, \cF\otimes Q)=\chi(Z, \cF\otimes Q)=\chi(Z, \cF)=1$, where the second equality uses the assumption that $Q$ is general.
Since $Z^\flat_1$ is of general type and  by Lemma \ref{basic}, $\cF_Q$ is a Hodge sheaf on $B^\flat_1$ supported on $Z^\flat_1$, $\cF_Q$ is M-regular.
By upper-semi-continuity of cohomology, $h^0(Z^\flat_1, \cF_Q\otimes Q')=1$ for $Q'\in \Pic^0(B^\flat_1)$ general.
Hence, $\chi(Z^\flat_1, \cF_Q)=h^0(Z^\flat_1, \cF_Q\otimes Q')=1$.

Since $Q\notin V^1(\cF)$, we know that $R^1h^\flat_{1*}(\cF\otimes Q)=0$.
Otherwise, $R^1h^\flat_{1*}(\cF\otimes Q)$ is a M-regular sheaf on $Z^\flat_1$ since $Z^\flat_1$ is of general type.
Then $V^0(R^1h^\flat_{1*}(\cF\otimes Q))=\widehat{B^\flat_1}$ and by
$(P3)$, $Q\in V^1(\cF)$, which is a contradiction. Hence, $$h^1(Z, \cF\otimes Q\otimes {h^\flat_1}^*P)=h^1(Z^\flat_1, \cF_Q\otimes P),$$ for any $P\in \widehat{B^\flat_1}$.

Consider the sequence
\begin{eqnarray*}
\xymatrix{\widehat{B^\flat_1}\ar@{^{(}->}[r]^{{h^\flat_1}^*} &\PB\ar@{->>}[r]^q &\PB/\widehat{B^\flat_1}}.
 \end{eqnarray*}
We then conclude that $V^1(\cF_Q)+Q=q^{-1}(q(Q))\cap V^1(\cF)$. As $Q \in \widehat{B^\flat_1}$ is general, we have the  bijection $S_2\rightarrow S_1$ described above.
\end{proof}

\begin{lemm}\label{complement}
Let $Z\hookrightarrow B$ be a subvariety of general type
of codimension $s\geq 2$. Assume that $Z$ generates $B$ and there exists a Hodge sheaf $\cF$ supported on $Z$ with $\chi(Z, \cF)=1$. For any two components $\PT_i=P_i+PB_i$ and $\PT_j=P_j+PB_j$, if $\PB_i+\PB_j=\PB$,
then $\widehat{B^\flat_i}+\widehat{B^\flat_j}=\PB$.
\end{lemm}

\begin{proof}
We denote by $\widehat{B_{ij}}$ the neutral component of
$\PB_i\cap \PB_j$ and let $Z_{ij}$ be
the image of $Z\hookrightarrow B\rightarrow B_{ij}$. Then,  the induced morphism $B\rightarrow B_i\times_{B_{ij}}B_j$ is an isogeny since $\PB_i+\PB_j=\PB$. We have the following commutative diagram:
\begin{eqnarray*}
\xymatrix{
Z\ar@{^{(}->}[r]\ar@{->}[d] & B\ar[d]^{\mathrm{isogeny}}\\
Z_i\times_{Z_{ij}}Z_j\ar@{^{(}->}[r]& B_i\times_{B_{ij}}B_j.}
\end{eqnarray*}

Hence $Z$ is an irreducible component of an \'etale cover of $Z_i\times_{Z_{ij}}Z_j$.

 Let $R_{i}$ (resp. $R_{ij}$) be the $\kappa$-kernel of $Z_i$ (resp. $Z_{ij}$).
 Since  $Z$ is of general type,
a general fiber of $Z \to Z_j$  is of general type, and hence the general fiber of $Z_i\rightarrow Z_{ij}$ is also of general type.

 Then the composition of morphisms $R_{i}\rightarrow B_i\rightarrow B_{ij}$
is an isogeny onto its image $R_{i}'$ and  $R_i'+Z_{ij}=Z_{ij}$. We denote by $B_{ij}'$ the quotient $B_{ij}/R_i'$ and $Z_{ij}'=Z_{ij}/R_i'$. Note that $B_i^\flat=B_i/R_i$. Hence
$$\begin{array}{ll} \dim {B_i^\flat} + \dim { B_j} -\dim {B'_{ij}}  &= \dim B_i-\dim R_i+\dim B_j-(\dim B_{ij}-\dim R_{i}') \\
& =\dim B_i+\dim B_j-\dim B_{ij}=\dim B.\end{array}$$

It follows that  the natural surjective morphism $B\rightarrow B^\flat_i\times_{B'_{ij}}B_j$ is again an isogeny and $Z$ is an irreducible component of the inverse image under this isogeny of $Z^\flat_i\times_{Z'_{ij}} Z_j$. Hence $\widehat{B^\flat_i}+\PB_j=\PB$.

 We  apply the same argument to $B_j$ and $B^\flat_i$.
 We then conclude as before that the morphism $R_j\rightarrow B'_{ij}$ is also an isogeny onto its image.  In particular, $\widehat{B^\flat_i}+\widehat{B^\flat_j}=\PB$.
\end{proof}

\begin{prop}\label{components}
Let $Z\hookrightarrow B$ be a subvariety of general type
of codimension $s\geq 2$. Assume that $Z$ generates $B$ and there exists a Hodge sheaf $\cF$ supported on $Z$ with $\chi(Z, \cF)=1$. Then,
there exist at least $s$ components $T_i=P_i+\PB_i$ of $V^1(\cF)$, $1\leq i\leq s$,
such that $\PB_i+\PB_j=\PB$ for any $i\neq j$. We call this collection of components the essential components of $V^1(\cF)$.
\end{prop}
\begin{proof}
Note that $\codim_{B^\flat_1}Z^\flat_1=s-1$ by Lemma \ref{fibration}, $Z^\flat_1$  generates $B^\flat_1$, and $\chi(Z^\flat_1, \cF_Q)=1$ for the Hodge sheaf $\cF_Q$ as in Lemma \ref{correspondence}.

 We run induction on $s$. When $s=2$, we know that $V^1(\cF_Q)\neq \emptyset$, hence by the correspondence in Lemma \ref{correspondence}, there exists $T_2$ such that $q(\PB_2)=\PB/\widehat{B_1^\flat}$. Then $\PB_1+\PB_2=\PB$.

When $s\geq 3$, by induction, there exist $s-1$ essential components of $V^1(\cF_Q)$ corresponds to subvarieties: $\widehat{B^1_2}, \ldots, \widehat{B^1_s}$ such that $ \widehat{B^1_i} +\widehat{B^1_j} =\widehat{B^\flat_1}$ for any $2 \le i \ne j \le s$.

%

By the  bijection in Lemma \ref{correspondence} , there exist correspondingly $s-1$ components of $V^1(\cF)$: $T_2=P_2+\PB_2,\ldots, \PT_s=P_s+\PB_s$.
Since $q(\PB_i)=\PB/\widehat{B_1^\flat}$ for each $2\leq i\leq s$, $\PB_i+\widehat{B_1^\flat}=\PB$ and hence $\PB_i+\PB_1=\PB$.  Moreover, as $\widehat{B^1_i} +\widehat{B^1_j} =\widehat{B^\flat_1}$, one has
$\PB_i+\PB_j=\PB$ for  $2\leq i<j\leq s$.
\end{proof}

\begin{theo}\label{subvariety}
Let $Z\hookrightarrow B$ be a subvariety of general type
of codimension $s\geq 2$.
Assume that $Z$ generates $B$ and that there exists a Hodge sheaf $\cF$ supported on $Z$ with $\chi(Z, \cF)=1$.
Fix $s$ essential components
$T_i=P_i+\PB_i$ of $V^1(\cF)$, $1\leq i\leq s$ satisfying $\PB_i+\PB_j=\PB$ for any $i\neq j$.
Let $\widehat{U_i}={\big(\bigcap_{j\neq i}\widehat{B_j}\big)_0}$ be the neutral component of
$\big(\bigcap_{j\neq i}\widehat{B_j} \big)$.
Let $K_i:=U^\flat_i$ (c.f. the notations in Section \ref{notation}).
Then we have the followings:
\begin{itemize}
\item[(1)]  $\sum_{1\leq j\leq s}\PK_j=\PB$ and $\sum_{j\neq i}\PK_j=\widehat{B^\flat_i}$ for each $i$;
\item[(2)] the image of the composition of morphisms $Z\hookrightarrow B\twoheadrightarrow K_i$ is an ample divisor $D_i$ of $K_i$ for each $i$.
\end{itemize}
\end{theo}
\begin{proof}
We will prove by induction on $s\geq 2$. If $s=2$, $U_1=B_2$ and $U_2=B_1$.
Hence $K_1=B^\flat_2$  and $K_2=B^\flat_1$ then we are done.  

 We assume that the statement of Theorem \ref{subvariety} holds when the codimension of the subvariety in abelian variety is at most $s-1$. As in Lemma \ref{correspondence}, we consider $\cF_Q^i=h^\flat_{i*}(\cF\otimes Q)$, for a general torsion $Q\in \PB$. Then $\cF_Q^i$ is a Hodge sheaf on $Z_i^\flat$ with $\chi(Z_i^\flat, \cF_Q^i)=1$.

Indeed, for any $i$,
by the bijection in Lemma \ref{correspondence}, each $T_j$ for $j\neq i$ corresponds to a component $T_j^i=P_j^i+\widehat{B_j^i}$ of $V^1(\cF_Q^i)$. 
Since $\PB_j+\PB_k=\PB$ for any $j\neq k$, we have $\PB_j^i+\PB_k^i=\widehat{B^\flat_i}$, for any $i$, $j$, $k$ pairwise distinct.

Since $\codim_{B^\flat_i}Z^\flat_i=s-1$, by induction, for each $t\neq i$, let
 $$\widehat{U_t^i}:={\big(\bigcap_{j\neq i, t}\widehat{(B_j^i)} \big)_0},$$ and $K^i_t=(U^i_t)^\flat$.
 Then by induction hypotheis, one has
\begin{itemize}
 \item[(1')]  $\sum_{t\neq i}\widehat{K_t^i}=\widehat{B^\flat_i}$;
 \item[(2')] the image of the composition of morphisms $Z\hookrightarrow B\twoheadrightarrow K_t^i$ is an ample divisor.
\end{itemize}
On the other hand, we have
\begin{eqnarray*}
\widehat{U_t^i} ={\big(\bigcap_{j\neq i, t}\widehat{B_j^i}\big)_0}
= {\big((\bigcap_{j\neq i, t}\widehat{B_{j}})\cap \PB_i^\flat \big)_0}.
\end{eqnarray*}
Since $\widehat{U_t}={\big(\bigcap_{j\neq t}\widehat{B_j}\big)_0}=\big((\bigcap_{j\neq i, t}\widehat{B_{j}})\cap \PB_i\big)_0$,
by Lemma \ref{easy}, we have
\begin{eqnarray*}
 K^i_t=(U_t^i)^\flat=U_t^\flat=K_t.
\end{eqnarray*}

Hence, by
$(1')$, $\sum_{j\neq i}\PK_j={B^\flat_i}$ for each $j\neq i$. Then $\sum_{1\leq j\leq s}\PK_j=\PB$. We deduce $(2)$ from $(2')$.
 \end{proof}

\begin{rema}\label{remark}
 We have actually proved that $Z$ is an irreducible component of an \'etale cover of certain fibre product of the ample divisors $D_i$ of $K_i$, $1\leq i\leq s$.
 \end{rema}

 In general, we can not expect that $\dim (\PK_i\cap \widehat{B^\flat_i})=0$ for $1\leq i\leq s$ and that $Z$ is an \'etale cover of the product of these $D_i$.
 However, we have this nice picture in some special cases.
 \begin{lemm}\label{dim=2}
  Under the assumption of Theorem \ref{subvariety}, if $\dim K_i=2$, then $\dim (\PK_i\cap \widehat{B^\flat_i})=0$ and hence we have a commutative diagram
  \begin{eqnarray*}
\xymatrix{
 Z\ar@{^{(}->}[r]\ar@{->>}[d] & B\ar[d]^{\rho}_{\mathrm{isogeny}}\\
 D_i\times Z_i \ar@{^{(}->}[r] & K_i\times B_i.}
\end{eqnarray*}
 \end{lemm}
\begin{proof}
 We know that $\PK_i+\widehat{B^\flat_i}=\PB$ by Theorem \ref{subvariety} and $\codim_{\PB}\PB_i\geq 2$. If $\dim \PK_i=\dim K_i=2$, then $\widehat{B^\flat_i}=\PB_i$
 and $\dim (\PK_i\cap \widehat{B^\flat_i})=0$.
\end{proof}

\begin{coro}\label{dimension}
Under the assumption of Theorem \ref{subvariety}, we have $2\dim Z\geq \dim B$.

Assume that $2\dim Z = \dim B$.
Pick $s=\codim_B Z$ essential components $\PT_i$ of $V^1(\cF)$.
For each $\PK_i$ defined in Theorem \ref{subvariety}, we have $\dim \PK_i=2$  and we have a commutative diagram:
\begin{eqnarray}\label{extremecase1}
\xymatrix{
 Z\ar@{^{(}->}[r]\ar@{->>}[d] & B\ar[d]^{\rho}_{\mathrm{isogeny}}\\
 D_1\times\cdots\times D_s \ar@{^{(}->}[r] & K_1\times\cdots\times K_s.}
\end{eqnarray}

\end{coro}

\begin{proof}
We prove by induction on $s$. If $s=1$, since $V^1(\cF)\neq \emptyset$ and a codimension-$(s+1)$
component of $V^1(\cF)$ is a component of $V^s(\cF)$,
we conclude that $B$ is an abelian surface and we are done.

In general, we take $\PT_1=P_1+\PB_1$ a component of $V^1(\cF)$. Assume that $\dim Z_1^\flat=n-s$ and $\dim B_1^\flat=\dim Z-s-1$
for $s\geq 1$.
As for $Q\in \PB$ general, $\cF_Q^1=h_{1*}^\flat(\cF\otimes Q)$ is a Hodge sheaf on $Z_1^\flat$ with $\chi(Z_1^\flat, \cF_Q^1)=1$.
By induction, one has $$\dim Z_1^\flat \ge \textrm{codim}_{B_1^\flat} Z_1^\flat = \textrm{codim}_{B_1} Z_1 =s-1,$$ where the last equality follows from Lemma \ref{fibration}. Since $\dim Z > \dim Z_1^\flat$, one has $\dim Z \ge s$ immediately.

Assume that $\dim B=2s$.
For the $s$ components $\PT_i=P_i+\PB_i$ of $V^1(\cF)$, by the same argument as above,
we see that $\dim Z_i^\flat=s-1$ and $\dim B_i^\flat=\dim B_i=2s-2$.
Hence, by induction, each $Z_i^\flat$ has the structure as in (\ref{extremecase1}).
Hence $\dim K_i=2$ and $\dim(\PK_i\cap \PB_i)=0$  for each $1\leq i\leq s$.
We then have the diagram (\ref{extremecase1}).
\end{proof}

\begin{coro}\label{inequality}
Use the same assumptions of Theorem \ref{subvariety} and furthermore assume that
$B$ has $k$ simple factors. Then $s=\codim_BZ\leq k$.
If $s=\codim_BZ= k$, we have a commutative diagram:
\begin{eqnarray}\label{extremecase}
\xymatrix{
 Z\ar@{^{(}->}[r]\ar@{->>}[d] & B\ar[d]^{\rho}_{\mathrm{isogeny}}\\
 D_1\times\cdots\times D_s \ar@{^{(}->}[r] & K_1\times\cdots\times K_s,}
\end{eqnarray}
where $D_i\hookrightarrow K_i$ is an ample divisor for each $1\leq i\leq s$.
In particular $Z$ is an irreducible component of $\rho^{-1}(D_1\times\cdots\times D_s)$.
\end{coro}

\begin{proof}
We argue by induction on $k$. If $k=1$, all components of $V^1(\cF)$ are isolated points. By Lemma \ref{fibration},
$Z$ is a divisor of $B$.
In general, if $s>2$, we consider $Z_1^\flat\hookrightarrow  B_1^\flat$ with the Hodge sheaf $\cF^1_Q$ such that
$\chi(Z_1^\flat, \cF^1_Q)=1$. Note that
$\codim_{B_1^\flat}Z_1^\flat=s-1$ and $B_1^\flat$ has at most $k-1$ simple factors. Hence by induction, $s\leq k$.

If the equality holds, each $K_i$ is a simple abelian variety and $\dim (\PK_i\cap \sum_{j\neq i}\PK_j)=0$.
Then the natural morphism $\rho: B\rightarrow K_1\times\cdots\times K_s$ is an isogeny and we have (\ref{extremecase}).
\end{proof}
We then finish the proof of Theorem \ref{fibration-to-divisor}.
\begin{proof}
Let $X$ be a compact K\"ahler manifold with $p_g=1$. We consider the diagram (\ref{setting}). Note that
$f_*\omega_X$ is a Hodge sheaf on $B$ supported on $Z_X$ with $h^0(Z_X, f_*\omega_X)=p_g(X)=1$. Hence $\chi(Z_X, \cF)=1$. Then Theorem \ref{fibration-to-divisor}
 follows easily from Theorem \ref{subvariety} and Corollary \ref{dimension}.
\end{proof}

 \section{Strong Hodge sheaf $\cF$ supported on $Z$ with $\chi(\cF)=1$ and theta divisors}
The main goal of this section is to study the following problem, which is a generalization of question asked by Pareschi (see \cite[Question 4.6]{JLT}).

\begin{qu}\label{problem}
Let $Z\hookrightarrow B$ be a subvariety of an abelian variety and $Z$ generates $B$.
Assume that $Z$ is of general type and
there exists a {\it strong Hodge sheaf} $\cF$ on $Z$ such that $\chi(Z, \cF)=1$. Then does  there exist theta divisors
$\Theta_i$, $1\leq i\leq m$ and a birational morphism $t: Z':=\Theta_1\times\cdots\times \Theta_m\rightarrow Z$ such that $\cF=t_*(\omega_{Z'}\otimes Q)$ for some torsion line bundle $Q$ on $Z'$?
\end{qu}

The main results in \cite{JLT} state that we have a positive answer to the above question in  two special cases:
\begin{itemize}
\item[(1)] $\cF$ is the pushforward of the canonical sheaf of  a desingularization of $Z$ and $Z$ is smooth in codimension $1$;
\item[(2)] $\dim Z=\frac{1}{2}\dim B$ and $\cF=f_*\omega_X$, where $f: X\rightarrow Z$ is a surjective morphism from a smooth projective variety $X$ to $Z$.
\end{itemize}

Results in Section $3$ provide further evidences for a positive answer
and we will prove it in some other cases.

 \begin{lemm}\label{curve}
Assume that $Z$ is a smooth projective curve of genus at least $2$ and $\cF$ is a torsion-free sheaf on $Z$ such that $\cF\otimes \omega_Z^{-1}$ is weakly positive.  Then $\chi(Z, \cF)\geq 1$. If $\chi(Z, \cF)=1$, then $Z$ is a smooth projective curve of genus $2$ and $\cF=\omega_Z\otimes Q$ for some line bundle $Q\in \text{Pic}^0(Z)$.
If $\cF$ is a Hodge sheaf and $h^0(C, \cF)=1$, then $Q$ is a torsion line bundle.
\end{lemm}
\begin{proof}
Since $\cF\otimes \omega_Z^{-1}$ is weakly positive on the smooth projective curve $Z$, it is nef.
In particular $\deg \cF \geq r(2g-2)$, where $r$ is the rank of $\cF$.
By the Riemann-Roch formula, $\chi(C, \cF) \geq r(g-1) \geq 1$.
If equality holds, $r=1, g=2, \deg \cF=2$.

If $\cF$ is a Hodge sheaf, then the cohomology support loci is a union of torsion translates of abelian subvarieties of $J(C)$.
So $Q$ has to be a torsion line bundle.
\end{proof}

\begin{theo}\label{products-curves}
Under the assumption of Question (\ref{problem}), assume moreover that $\dim B=2\dim Z$, then  Question (\ref{problem}) has an affirmative answer.
\end{theo}
\begin{proof}
We apply Corollary (\ref{inequality}) and diagram (\ref{extremecase1}).  Let
$C_i$ be the normalization of $D_i$ and $Z'$ the connected component of
$(C_1\times\cdots\times C_n)\times_{(K_1\times\cdots\times K_n)}B$ which dominates $Z$. We then have
\begin{eqnarray*}
\xymatrix{
 Z'\ar[d]_{\pi}\ar[r]^{\tau}& Z\ar@{^{(}->}[r]\ar@{->>}[d] & B\ar[d]^{\rho}_{\mathrm{isogeny}}\\
C_1\times\cdots\times C_n\ar[r]\ar[d]& D_1\times\cdots\times D_n \ar@{^{(}->}[r]\ar[d] & K_1\times\cdots\times K_n\ar[d]\\
C_1\ar[r]^{\tau_1} & D_1\ar@{^{(}->}[r] & K_1,}
\end{eqnarray*}
where $\pi$ is an \'etale morphism and $\tau$ is a desingularization.

We prove by induction on $\dim Z$. If $\dim Z=1$, then there exists $\cF'$ on $Z'$ such that $\R\tau_*\cF'=\cF$ and $\cF'\otimes\omega_{Z'}^{-1}$ is weakly positive. By Lemma \ref{curve}, we conclude the proof.

We then assume that Theorem \ref{products-curves} holds in dimension $<n$. Take $Q\in \PB$ a general torsion point, for the morphism $t_1: Z\rightarrow D_1$, let
$\cQ_Q=t_{1*}(\cF\otimes Q)$. Then $\cQ_Q$ is a strong Hodge sheaf on $K_1$ supported on $D_1$. Moreover,
 $h^0(D_1, \cQ_Q)=1$ and  hence $\chi(D_1, \cQ_Q)=1$. By $(P5)$, there exists $\cQ_Q'$ on $C_1$ such that $ \R\tau_{1*}\cQ_Q'=\cQ_Q$ and
 $\cQ_Q'\otimes \omega_{C_1}^{-1}$ is weakly positive.
 Then by Lemma \ref{curve}, $C_1$ is a smooth curve of genus $2$ and
 $\cQ_Q=\tau_{1*}(\omega_{C_1}\otimes Q_1)$ for a torsion line bundle $Q_1\in \Pic^0(C_1)$. In particular, $\rank \cQ_Q=1$. Similarly, we see that each $C_i$ has genus $2$.

 Let $F$ be a general fiber of $t_1$ and let $B_F$ the correspondingly fiber of $B\rightarrow K_1$.
 Then $\cF_Q':=(\cF\otimes Q)\mid_F$ is a strong Hodge sheaf on $B_F$ supported on $F$ such that $h^0(F, \cF_Q')=1$. By induction, $F$ is birational to a product of $n-1$ genus $2$ curves.  Then the natural morphism $F\rightarrow D_2\times \cdots\times D_s$ is birational and so is the morphism $Z\rightarrow D_1\times\cdots\times D_s$. Hence $\pi$ is an isomorphism.

 It remains to show that $\cF'$ is isomorphic to the canonical bundle of $Z'$ twisted by a torsion line bundle. Note that $\cF'$ restricted to each factor $C_i$ is of such form by Lemma \ref{curve} and we conclude by the see-saw principle.
\end{proof}
\begin{theo}\label{thetas}
Let $Z$ be a subvariety of general type of an abelian variety $B$ of codimension $s$.
Assume that  there exists a strong Hodge sheaf $\cF$ on $B$ supported on $Z$ such that $\chi(Z, \cF)=1$, $Z$ is smooth in codimension $1$  and $B$ has $s$ simple factors.
Then Question \ref{problem} has an affirmative answer.
\end{theo}

We first prove the divisorial case and then apply  Theorem \ref{subvariety} to conclude the proof for the general case.

\subsection{Divisorial case}\label{divisorcase}
In this subsection, we assume that $B$ is a simple abelian variey of dimension $g$, $Z$ is an irreducible ample divisor smooth in codimension $1$ and there exists a strong Hodge sheaf $\cF$ on $B$ such that $\chi(Z, \cF)=1$.  Note that in this case $Z$ is normal.

We aim to prove Theorem 4.3 in this case. Our argument follows closely that of Pareschi in \cite[Theorem 5.1]{P}.

Let $\rho: X\rightarrow Z$ be a desingularization and by definition of stong Hodge sheaf,  there exists a torsion-free sheaf  $\cF'$ on $X$ such that $\R\rho_*\cF'=\rho_*\cF'=\cF$ and $\cF'\otimes \omega_X^{-1}$ is weakly positive.   Since $Z$ is normal,  the composition of morphism $X\xrightarrow{\rho} Z\hookrightarrow B$ is primitive, namely the pull-back $\PB\rightarrow \Pic^0(X)$ is injective  (in other words, any \'etale over of $Z$ induced by an \'etale cover of $B$ remains irreducible).

Let $\R\Delta_B(\cF)=\R\mathcal{H}om(\cF, \cO_B)$. Since $\cF$ is M-regular on $B$ and $\chi(\cF)=1$,
$$\R\Phi_{\cP_B}(\R\Delta_B(\cF))=(L\otimes\cI_V)[-g],$$
where $V$ is a subscheme of $\PB$ of dimension zero and $L$ is a line bundle on $\PB$.
Since $B$ is simple, $L$ is either ample, anti-ample, or trivial.
We will see soon that $L$ is ample.

By Fourier-Mukai equivalence,
\begin{eqnarray}\label{isom}
\R\Delta_B(\cF)\simeq (-1_B)^* \R\Psi_{\cP_B}(L\otimes \cI_V),
\end{eqnarray}
and both are supported on $B$.

We then consider the short exact sequence on $\PB$:
\begin{eqnarray*}
0\rightarrow L\otimes \cI_V\rightarrow L\rightarrow L\mid_V\rightarrow 0,
\end{eqnarray*}
and apply the functor $\R\Psi_{\cP_{B}}$:
\begin{eqnarray*}
0\rightarrow\R^0\Psi_{\cP_{B}}(L\otimes \cI_V)\rightarrow \R^0\Psi_{\cP_{B}}(L)\rightarrow \R^0\Psi_{\cP_{B}}(L\mid_V)\rightarrow \R^1\Psi_{\cP_{B}}(L\otimes \cI_V)\rightarrow \R^1\Psi(L).
\end{eqnarray*}

From this long exact sequence, we know that $L$ has sections otherwise $\R^0\Psi_{\cP_{B}}(L\mid_V)\rightarrow \R^1\Psi_{\cP_{B}}(L\otimes \cI_V)$ is an injection and $\R\Delta_B(\cF)$ cannot be supported on $B$.
If $L$ is trivial, then $\R^0\Psi_{\cP_{B}}(L\otimes \cI_V)=0$, $\R^0\Psi_{\cP_{B}}(L)$ is the skyscraper sheaf supported at $0 \in B$, and $\R^0\Psi_{\cP_{B}}(L\mid_V)$ is (always) locally free.
Thus this does not happen either.
That is, $L$ is ample and  $\R^i\Psi_{\cP_{B}}(L\otimes \cI_V)\simeq \R^i\Psi_{\cP_{B}}(L)=0$ for $i>1$.
We can write the previous long exact sequence more explicitly as
\begin{eqnarray*}
0\rightarrow\R^0\Psi_{\cP_{B}}(L\otimes \cI_V)\rightarrow \R^0\Psi_{\cP_{B}}(L)\rightarrow \R^0\Psi_{\cP_{B}}(L\mid_V)\rightarrow \R^1\Psi_{\cP_{B}}(L\otimes \cI_V)\rightarrow 0.
\end{eqnarray*}

Since $\R^0\Delta_B(\cF)=0$, we have $\R\Psi_{\cP_B}(L\otimes \cI_V)=\R^1\Psi_{\cP_B}(L\otimes \cI_V)[-1]$ and
\begin{eqnarray}\label{short}0\rightarrow \R^0\Psi_{\cP_{B}}(L)\rightarrow \R^0\Psi_{\cP_{B}}(L\mid_V)\rightarrow \R^1\Psi_{\cP_{B}}(L\otimes \cI_V)\rightarrow 0.\end{eqnarray}
Taking the duality of (\ref{isom}), we see that $\cF\simeq \mathcal{E}xt^1((-1_B)^* \R^1\Psi_{\cP_{B}}(L\otimes \cI_V), \cO_{\PB})$. Apply the functor $\R\Delta_{B}$ to (\ref{short}), we have:
\begin{eqnarray*}
0\rightarrow \R^0\Psi_{\cP_{B}}(L\mid_V)^{\vee}\rightarrow \R^0\Psi_{\cP_{B}}(L)^{\vee}\rightarrow \cF\rightarrow 0.
\end{eqnarray*}

We know that $W:=\R^0\Psi_{\cP_{B}}(L\mid_V)^{\vee}$ is a flat vector bundle and $\R^0\Psi_{\cP_{B}}(L)^{\vee}$ is an ample vector bundle on $B$.

Let $\psi_L: \PB \rightarrow B$ be the isogeny induced by $L$.
Assume that $h^0(\PB, L)=k\geq 1$. It follows that $\deg \psi_L=k^2$.
Moreover, we know that $\psi_L^*\R^0\Psi_{\cP_{B}}(L)^{\vee}=L^{\oplus k}$. Let $\widetilde{W}=\psi_L^*W$ and $\widetilde{\cF}=\psi_L^*\cF$.
As a consequence,
\begin{eqnarray}\label{sequence}
0\rightarrow \widetilde{W}\rightarrow L^{\oplus k}\rightarrow \widetilde{\cF}\rightarrow 0.
\end{eqnarray}

We compute
\begin{eqnarray*}
ch(\widetilde{\cF})&=&ch( L^{\oplus k})-ch(\widetilde{W})\\
&\equiv &k\sum_{m\geq 1}\frac{1}{m!}L^m,
\end{eqnarray*}
where $\equiv$ means algebraic equivalence between algebraic cycles on $\PB$.

On the other hand, since $Z\rightarrow B$ is primitive,  the induced \'etale cover $\widetilde{X}:=X\times_A\PA$ is irreducible and so is $\widetilde{Z}:=Z\times_A\PA$.
Let $\rho': \widetilde{X}\rightarrow \widetilde{Z}$ be the induced desingularization
and $\widetilde{\cF}'$ be the pullback of $\cF$ on $\widetilde{X}$.
We denote
\begin{eqnarray*}
\xymatrix{
i: \widetilde{X}\ar[r]^{\rho'} & \widetilde{Z}\ar[r]^j &\PB}.
\end{eqnarray*}
Thus
\[
\R i_*\widetilde{\cF}'=i_*\widetilde{\cF}'=\widetilde{\cF},
\]
 and
$\widetilde{\cF}'\otimes\omega_{\widetilde{X}}^{-1}$ is also weakly positive on $\widetilde{X}$.
By Grothendieck-Riemann-Roch, $$i_*(ch(\widetilde{\cF'})Td(\widetilde{X}))\equiv ch(\widetilde{\cF})\equiv k\sum_{m\geq 1}\frac{1}{m!}L^m.$$
Let $\rank \widetilde{\cF}=\rank \widetilde{\cF'}=k_1\leq k$.
Compute the degree $1$ terms, we have  $k_1\widetilde{Z}\equiv kL$. Compute the degree $2$ terms, we have
$$i_*\big(c_1(\widetilde{\cF'})-\frac{1}{2}kc_1(\omega_{\widetilde{X}})\big)\equiv \frac{1}{2}kL^2.$$

Since $\widetilde{\cF'}\otimes \omega_{X'}^{-1}$ is weakly positive, $D:=\det\widetilde{\cF'}-k_1K_{\widetilde{X}}$ is a psuedo-effective divisor on $\widetilde{X}$ (see for instance \cite[Corollary 2.20]{V}). Then
$i_*(\frac{1}{2}k_1K_{\widetilde{X}}+D)\equiv\frac{1}{2}kL^2$. Hence  we write $$j_*(\rho'_*K_{\widetilde{X}}+\rho'_*D')\equiv \frac{k}{k_1}L^2,$$ where $D'$ is a pseudo-effective $\mathbb{Q}$-divisor on $\widetilde{X}$.

Since $\wZ$ is normal, we have $\rho'_*K_{\wX}=K_{\wZ}=\cO_{\PB}(\wZ)\mid_{\wZ}$ and hence $[j_*\rho'_*(K_{\wX})]=(\frac{k}{k_1})^2[L]^2\in H^4(X, \mathbb{Q})$.  As $D$ is pseudo-effective, we see immediately that $k=k_1$ and $\cO_{\PB}(\wZ)$ is algebraically equivalent to $L$.  We note moreover that $\rho'_*\omega_{\wX}=\omega_{\wZ}\otimes \cI$ for some ideal sheaf $\cI$. Hence, for a general $Q\in\Pic^0(\PB)$,
$$\chi(\wX, \omega_{\wX})=h^0(\wX, \omega_{\wX}\otimes i^*Q)=h^0(\wZ, \rho'_*\omega_{\wX}\otimes Q)\leq h^0(\wZ, \cO_{\PB}(\wZ)\mid_{\wZ}\otimes Q)=k.$$
On the other hand, $\wX\rightarrow X$ is an \'etale cover of degree $k^2$. Thus, $k=1$, $\phi_L: \PA\rightarrow A$ is an isomorphism and $Z\simeq \wZ$ is a theta divisor. By $(\ref{sequence})$, $\cF=\omega_Z\otimes Q$ for some torsion line bundle $Q\in \PB$.

\begin{rema}
Instead of assuming that $Z$ is smooth in codimension $1$, we can conclude by a similar argument by simply assuming that  $\rho_*K_{X}\equiv M\mid_D$ for some line bundle $M$ on $B$.
\end{rema}
\subsection{General case}
We  consider the commutative diagram (\ref{extremecase}) in Corollary \ref{inequality} and argue by induction on $s$.

If $Z\hookrightarrow B$ is not primitive, we can take an \'etale cover of $B'\rightarrow B$ such that,  for an irreducible component $Z'$ of $Z\times_BB'$, $Z'\hookrightarrow B'$ is primitive and $Z'$ is birational to $Z$. Hence we  will assume that $Z\hookrightarrow B$ is primitive.

Since $D_i$ is the image of the natural morphism $Z\rightarrow K_i$, $D_i\hookrightarrow K_i$ is also primitive. Hence, $\rho^{-1}(D_1\times\cdots\times D_s)$ is irreducible and
$Z\simeq \rho^{-1}(D_1\times\cdots\times D_s)$. Since $Z$ is smooth in codimension $1$, each $D_i$ is smooth in codimension $1$.
Moreover each $D_i$ is an ample divisor of the simple abelian variety $K_i$.
  We denote by $p: Z\rightarrow D_1$ the natural morphism.  Then for $Q\in \Pic^0(B)$ general, we have
 $$\chi(D_1, p_*(\cF\otimes Q))=h^0(D_1, p_*(\cF\otimes Q))=h^0(Z, \cF\otimes Q)=1.$$

By (\ref{divisorcase}), each $D_i$ is a theta divisor and the sheaf $p_*(\cF\otimes Q)$ has rank $1$.
Then for a general fiber $F$ of $p$, $F$ is a subvariety of general type of $B_1:=\ker(B\rightarrow K_1)$ and is smooth in codimension $1$.
Note that $(\cF\otimes Q)\mid_F$ is a strong Hodge sheaf on $B_1$ supported on $F$.
Since $p_*(\cF\otimes Q)$ has rank $1$, $h^0(F, (\cF\otimes Q)\mid_F)=1$. It then follows that $\chi(F, \cF\otimes Q\mid_F)=1$.

By induction,  $F$ is birational to a product of theta divisors. Consider the induced morphisms
\begin{eqnarray*}
\xymatrix{
F\ar@{^{(}->}[r]\ar[d]^{\pi_F}&  B_1\ar[d]^{\pi}\\
D_2\times\cdots\times D_s\ar@{^{(}->}[r] & K_2\times\cdots\times K_s.}
\end{eqnarray*}
 Since $\pi$ is an isogeny, we see immediately that $\pi$ is an isomorphism and $\pi_F$ is also an isomorphism. Thus $Z\simeq D_1\times\cdots\times D_s$.  Moreover, $\cF$ is a torsion-free rank $1$ sheaf and by induction,
 $\cF\mid_{D_1\times y}\simeq \omega_{D_1}\otimes Q_1$ for all $y\in D_2\times\cdots\times D_s$, where $Q_1$ is a fixed torsion line bundle on $D_1$ and  $\cF\mid_{x\times  D_2\times\cdots\times D_s}\simeq \omega_{D_2\times\cdots\times D_s}\otimes Q_2$ for  all $x\in D_1$, where $Q_2$ is a fixed torsion line bundle on $D_2\times\cdots\times D_s$. Indeed, $Q_1$ and $Q_2$ can be read from the cohomological support loci of $\cF$.
 We then conclude that  $\cF=\omega_Z\otimes (Q_1\boxtimes Q_2)$.

\subsection{Proof of Theorem \ref{image-theta}}
Theorem \ref{image-theta} is a direct corollary of the proof of Theorem \ref{thetas}.
We have already proved that all the assertions except the last one that the Albanese map is a fibration.
To prove that the Albanese map is a fibration, simply note that in the first part of the proof, we have already established that $a_{X*}(\omega_X)$ has rank $1$.
Then the Albanese map has to be a fibration.

\section{Fibrations over genus $2$ curves}
In this section, we take into considerations  the map between the second Betti cohomology.

We first assume  Theorem \ref{subvariety1} and complete the proof of  Theorem \ref{fibration-to-curve}.
\begin{proof}We use the commutative diagram (\ref{setting}).

We first claim that the restriction map $H^2(B, \mathbb{Q})\rightarrow H^2(Z_X, \mathbb{Q})$ is not injective. Otherwise, since $Y=Z_X\times_{B}  A_X$ and $Z$ generates $B$,  the   map  $H^2(A_X, \mathbb{Q})\rightarrow H^2(Y, \mathbb{Q})$ is  also injective.
Moreover, since $g$ is a surjective, the pull-back  $g^*: H^i(Y, \mathbb{Q})\rightarrow H^i(X, \mathbb{Q})$  induces an
injective map $\mathrm{Gr}^W_iH^i(Y, \mathbb{Q})\rightarrow H^i(X, \mathbb{Q})$, where $W^{\cdot}$ is the weight filtration on
$H^*(Y, \mathbb{Q})$. Thus, $a_X^*: H^2(A_X, \mathbb{Q})\rightarrow H^2(X, \mathbb{Q})$ is injective and hence is a contradiction.

Thus, the restriction map $H^2(B, \mathbb{Q})\rightarrow H^2(Z_X, \mathbb{Q})$ is not injective and $Z_X$ generates $B$. By Theorem \ref{subvariety1} below, we have a fibration $h: \overline{Z}_X\rightarrow C$, where $\overline{Z}_X$ is the normalization of $Z_X$ and $C$ is a smooth projective curve of genus $2$. Since $f: X\rightarrow Z_X$ factors through the normalization of $Z_X$, we then have a fibration $\varphi: X\rightarrow C$.
 \end{proof}

\begin{theo}\label{subvariety1} Let $Z\hookrightarrow B$ be a subvariety of general type generating $B$ and let $\overline{Z}$ be the normalization of $Z$. Let $\cF$ be a strong Hodge sheaf on $Z$. Assume that $\chi(Z, \cF)=1$ and the restriction map $H^2(B, \mathbb{Q})\rightarrow H^2(Z, \mathbb{Q})$ is not injective. Then there exists a fibration $h: \overline{Z}\rightarrow C$ to a smooth projective genus $2$ curve $C$.
\end{theo}

\begin{proof}
We argue by induction on $\codim_BZ$.
If $\codim_BZ=1$,  $Z$ is an ample divisor of $B$.
By Lefschetz hyperplane theorem, the restriction map $H^i(B, \mathbb{Q})\rightarrow H^i(Z, \mathbb{Q})$ is injective for all $0\leq i\leq \dim Z$.
Hence $\dim Z=1$. By Lemma \ref{curve}, $\overline{Z}$ is a smooth projective curve of genus $2$.

In the following assume that $\codim_B Z=k\geq 2$ and that Theorem \ref{subvariety1} holds for subvarieties of $B$ whose codimension is less than $k$.

 Pick two components $\PT_1=P_1+\PB_1$ and $\PT_2=P_2+\PB_2$ of $V^1(\cF)$ such that $\PB_1+\PB_2=\PB$.
 Consider the morphisms $h_1^\flat: Z\rightarrow Z_1^\flat$ and $h_2: Z\rightarrow Z_2^\flat$  as in Lemma \ref{fibration}. \\\\

{\bf Claim:} \;{\em  Either $$\varphi_1: H^2(B_1^\flat, \mathbb{Q})\rightarrow  H^2(Z_1^\flat, \mathbb{Q})$$
is not injective or $$\varphi_2: H^2(B_2^\flat, \mathbb{Q})\rightarrow  H^2(Z_2^\flat, \mathbb{Q}
)$$ is not injective. }\\

 We argue by contradiction. Assume that both $\varphi_1$ and $\varphi_2$ are injective.
 Let $\PK$ be the neutral component of $\widehat{B_1^\flat}\cap \widehat{B_2^\flat}$. Then the induced morphism
 $B\rightarrow B_1^\flat\times_{K}B_2^\flat$ is an isogeny. We also take $B_1^\flat\rightarrow K_1'$ and $B_2^\flat\rightarrow K_2'$
 be quotients with connected fibers such that the indcued morphisms $B_i^\flat\rightarrow K_i'\times K$ are isogenies for $i=1$, $2$.

Note that

\begin{eqnarray*}H^2(B, \mathbb{Q})&=&\big(H^2(B_1^\flat, \mathbb{Q})+H^2(B_2^\flat, \mathbb{Q})\big)\oplus
\big(H^1(K_1', \mathbb{Q})\wedge H^1(K_2', \mathbb{Q})\big).
\end{eqnarray*}

 Let $0\neq \alpha=w +v$ in the kernel of the restriction map $\varphi: H^2(B, \mathbb{Q})\rightarrow H^2(Z, \mathbb{Q})$,
 where $w\in H^2(B_1^\flat, \mathbb{Q})+H^2(B_2^\flat, \mathbb{Q})$ and $v\in H^1(K_1', \mathbb{Q})\wedge H^1(K_2', \mathbb{Q})$.
 Moreover we can write $ H^2(B_i^\flat, \mathbb{Q})=W_i \oplus H^2(K, \mathbb{Q})$, where
 $W_i=H^2(K_i', \mathbb{Q})\oplus\big( H^1(K_i', \mathbb{Q})\wedge H^1(K, \mathbb{Q})\big)$.
 Then $$H^2(B_1^\flat, \mathbb{Q})+H^2(B_2^\flat, \mathbb{Q})=W_1\oplus  W_2\oplus H^2(K, \mathbb{Q}),$$
 and we suppose that $w=w_1+w_2+w_3$, where $w_i\in W_i$ for $i=1$, $2$ and $w_3\in H^2(K, \mathbb{Q})$.

 We then take a smooth models $Z'$ of $Z$, $Z_i'$ of $Z_i^\flat$ for $i=1$, $2$ and consider the maps
\begin{eqnarray*}
\xymatrix{Z'\ar[d]^{h_i'}\ar[r]^{\rho} & B\ar[d]^{p_i}\ar[r] & K_{3-i}'\\
Z_i'\ar[r]^{\rho_i}  & B_i^\flat.}
\end{eqnarray*}

Since $\varphi_i$ is injective and the Hodge structures on $H^2(B_i^\flat, \mathbb{Q})$ is pure,
$\rho_i^*: H^2(B_i^\flat, \mathbb{Q})\rightarrow H^2(Z_i', \mathbb{Q})$ is also injective, for $i=1$, $2$.
Note that $\alpha$ is also in the kernel of $\rho^*: H^2(B, \mathbb{Q})\rightarrow H^2(Z', \mathbb{Q})$.

We take a ample class $l\in H^2(B, \mathbb{Q})$. Let $s_i$ be the dimension of a general fiber of $p_i$.
Then $$0=h_{i*}'\rho^*((w+v)\cup l^{2s_i})=h_{i*}'\rho^*(w\cup l^{2s_i})=M_i\rho_i^*(w_i+w_3),$$ for some positive number $M_i$.
Since $\rho_i^*: H^2(B_i^\flat, \mathbb{Q})\rightarrow H^2(Z_i', \mathbb{Q})$ is  injective,
we conclude that $w_i=0$ for $i=1$, $2$, and $3$. Thus $w=0$.

Let $Z_3 $ be the image of the morphisms $Z\hookrightarrow A\rightarrow K$.  Since $Z$ generates $B$,
$Z_i$ generates $B_i$. Hence for a general fiber $F_i$ of $Z_i\rightarrow Z_3$, the natural map
$H^1(K_i', \mathbb{Q})\rightarrow H^1(F_i, \mathbb{Q})$ is injective. Let $F$ be a general fiber of $Z\rightarrow Z_3$.
Then we have natural morphisms $F\twoheadrightarrow F_1\times F_2\rightarrow K_1'\times K_2'$. Since the map
$H^1(K_1', \mathbb{Q})\wedge H^1(K_2', \mathbb{Q})\rightarrow H^2(F_1\times F_2, \mathbb{Q})$ is injective and
the Hodge structure on $H^1(K_1', \mathbb{Q})\wedge H^1(K_2', \mathbb{Q})$ is pure,
we conclude that the map $$H^1(K_1', \mathbb{Q})\wedge H^1(K_2', \mathbb{Q}) \rightarrow H^2(F, \mathbb{Q})$$
is also injective. This map factors through
$\varphi\mid_{H^1(K_1', \mathbb{Q})\wedge H^1(K_2', \mathbb{Q}) }: H^1(K_1', \mathbb{Q})\wedge H^1(K_2', \mathbb{Q}) \rightarrow H^2(Z, \mathbb{Q}) $, hence $\varphi\mid_{H^1(K_1', \mathbb{Q})\wedge H^1(K_2', \mathbb{Q}) } $ is also injective and $v=0$, which is a contradiction.\\ \\
{\bf  Conclusion:}

 We may assume that $\varphi_1$ is not injective. Moreover, $Z_1^\flat\hookrightarrow B_1^\flat$
 is a subvariety of general type and $Z_1$ generates $B_1$ and for $Q\in \PB$ a general torsion point,
 $\cF_Q:=h_{1*}^\flat(\cF\otimes Q)$ is a Hodge sheaf supported on $Z_1^\flat$ with $\chi(Z_1^\flat, \cF_Q)=1$.
 Hence by induction, there exists a fibration $\overline{Z_1^\flat}\rightarrow C$ to a smooth projective genus $2$ curve.
 Hence we have the induced fibration $h: \overline{Z}\rightarrow C$.
\end{proof}
\begin{theo}\label{split}
Under the assumption of Theorem \ref{subvariety1}, let $m=\dim \ker(H^2(B, \mathbb{Q})\rightarrow H^2(Z, \mathbb{Q}))$. Then $m$ is divisible by $5$. Let $m=5k$. Then there exists a fibration $\overline{Z}\rightarrow C_1\times\cdots\times C_k$, where $C_i$ is a smooth projective curve of genus $2$ for each $1\leq i\leq k$.
\end{theo}
\begin{proof}
  Let $s=\codim_BZ$. Take components $\PT_i=P_i+\PB_i$, $1\leq i\leq s$, of $V^1(\cF)$ such that $\PB_i+\PB_j=\PB$ for all $i\neq j$. Then, as in the proof of $\ref{subvariety1}$, by induction on $s$, we actually show that for some $K_i$ defined as in Theorem \ref{subvariety}, the map $H^2(K_j, \mathbb{Q})\rightarrow H^2(D_j, \mathbb{Q})$ is not injective for some $j$. Since $D_j$ is an ample divisor of $K_j$, we conclude that $\dim K_j=2$ and $D_j$ is a curve. Then by Lemma \ref{curve}, the normalization  $C_j$ of $D_j$ is a smooth projective curve of genus $2$. Moreover, by Lemma \ref{dim=2}, we have a commutative diagram:
  \begin{eqnarray*}
\xymatrix{
\overline{Z}\ar[rr]^{\mathrm{normalization}}\ar[d]_{\mathrm{abelian\; etale\; cover}}& &Z\ar@{^{(}->}[r]\ar@{->>}[d] & B\ar[d]^{\rho}_{\mathrm{isogeny}}\\
C_j\times\overline{Z_j}\ar[rr]^{\mathrm{normalization}} && D_j\times Z_j\ar@{^{(}->}[r] & K_j\times B_j.}
\end{eqnarray*}
Since $B$ is smooth, $W:=\ker(H^2(B, \mathbb{Q})\rightarrow H^2(Z, \mathbb{Q}))=\ker(H^2(B, \mathbb{Q})\rightarrow H^2(\overline{Z}, \mathbb{Q})) $. Moreover, $\rho$ is an isogeny, hence $W\simeq \ker(H^2(K_j\times B_j, \mathbb{Q})\rightarrow H^2(C_j\times \overline{Z_j}, \mathbb{Q}))$. Since $H^2(K_j\times B_j, \mathbb{Q})=H^2(K_j, \mathbb{Q})\oplus \big(H^1(K_j, \mathbb{Q})\wedge H^1(B_j, \mathbb{Q})\big) \oplus H^2(B_j, \mathbb{Q})$. As $Z$ generates $B$, we conclude that $$W\simeq \ker(H^2(K_j, \mathbb{Q})\rightarrow H^2(C_j, \mathbb{Q}))\bigoplus \ker(H^2(B_j, \mathbb{Q})\rightarrow H^2(Z_j, \mathbb{Q})).$$
Hence $\dim W=5+\dim  \ker(H^2(B_j, \mathbb{Q})\rightarrow H^2(Z_j, \mathbb{Q}))$. Since $Z_j\hookrightarrow B_j$ also satisfies the assumption of Theorem \ref{subvariety1}. We use induction to construct the morphism $\bar{Z} \to C_1 \times \ldots \times C_k$.
To see this map is a fibration, it suffices to show that a general fiber is connected. For this purpose, one only need to show that push forward of $\cF$ has rank $1$, or equivalently, the restriction of $\cF$ to a general fiber is one dimensional global sections.
One can prove this using induction. For the morphism $f_1: \bar{Z} \to C_1$, this is done in Lemma \ref{curve}.
Then we can work with a general fiber $F$ of $f_1$, which is of general type, maps to the product $C_2 \times \ldots \times C_k$ and carries a strong Hodge sheaf $\cF|_F$ with $h^0(F, \cF)=1$.
\end{proof}

With this observation it is very easy to prove Corollary \ref{cor:deRham}.
 \begin{proof}[Proof of Corollary \ref{cor:deRham}]
For a general discussion of the de Rham fundamental group, we refer the readers to \cite{KahlerGroup}.
For our purpose, it suffices to know that if the pull-back on cohomology of a morphism $f: X \to Y$ between smooth compact K\"ahler manifolds
induces an isomorphism $f^*: H^1 (Y, \mathbb{Q}) \to H^1(X, \mathbb{Q})$ and an injection $f^*: H^2(Y, \mathbb{Q}) \to H^2(X, \mathbb{Q})$,
then $f_*$ induces an isomorphism on de Rham fundamental groups.
A direct consequence of this observation is a result of Campana that a resolution of singularities of the Albanese image of a compact K\"ahler manifold $X$ computes the de Rham fundamental group $\pi_1(X) \otimes \mathbb{Q}$.

Let $5s=\dim \ker(H^2(A_X, \mathbb{Q})\rightarrow H^2(Z_X, \mathbb{Q}))$. By Theorem \ref{split} and its proof, we have a commutative diagram
 \begin{eqnarray*}
\xymatrix{
\overline{Z}\ar[rr]^{\mathrm{normalization}}\ar[d]_{\mathrm{abelian\; etale\; cover}}& &Z\ar@{^{(}->}[r]\ar@{->>}[d] & B\ar[d]^{\rho}_{\mathrm{isogeny}}\\
C_1\times\cdots \times C_s\times \overline{Z'}\ar[rr]^{\mathrm{normalization}} && D_1\times \cdots\times D_s\times Z'\ar@{^{(}->}[r] & K_1\times\cdots\times K_i\times B',}
\end{eqnarray*}
where $C_i$ is a smooth projective curve of genus $2$ for each $1\leq i\leq s$ and the map $H^2(B', \mathbb{Q})\rightarrow H^2(Z', \mathbb{Q})$ is injective.

We apply the above observation to a resolution of singularities of $Z$ and $Z'$, $\tilde{Z} \to C_1 \times \cdots \times C_s\times \tilde{Z'}$.
Since $\bar{Z}$ is an abelian \'etale cover of the product $ C_1 \times \cdots \times C_s\times  \tilde{\bar{Z'}}$, the induced map is an isomorphism on $H^1$ and injective on $H^2$ (for the resolutions).
 Hence $\pi_1(X)\otimes \mathbb{Q}\simeq \pi_1(\tilde{Z})\otimes \mathbb{Q}\simeq (\pi_1(C_1)\otimes \mathbb{Q})^s\times (\pi_1(\tilde{Z'})\otimes \mathbb{Q})\simeq  (\pi_1(C_1)\otimes \mathbb{Q})^s\times (\pi_1(B')\otimes\mathbb{Q}) $.

 \end{proof}

\section{Fake tori}
In this section, we will always assume that $X$ is a fake torus of dimension $n$ and consider the commutative diagram (\ref{setting}), with $Z$ replaced by $Z_X$ in the following, namely:
\begin{eqnarray*}
\xymatrix{X\ar@/_2pc/[dd]_f\ar[dr]^{a_X}\ar[d]_g\\
Y\ar[d]_h\ar@{^{(}->}[r] & A_X\ar[d]^p\\
Z_X\ar@{^{(}->}[r]  &B,}
\end{eqnarray*}
Note that $Y=Z_X\times_BA_X$. We summarize what we know about  $Z_X$. Let $s=\codim_{A_X}Y=\codim_BZ_X$. Note that since $Y$ is the Albanese image of $X$,
both $Y\hookrightarrow A_X$ and $Z_X\hookrightarrow B$
are primitive. By Theorem \ref{subvariety1}, Remark \ref{split}, and Corollary \ref{products-curves}, we have:
\begin{itemize}

\item[(1)]  $\dim B\geq 2s$. If equality holds, $Z_X\simeq C_1\times\cdots \times C_s$, where each $C_i$ is a smooth projective curve of genus $2$ and $f_*\omega_X=\omega_Z\otimes Q$ for some torsion line bundle
$Q\in \PB$;
 \item[(2)] if $s=1$, then $Z_X$ is a smooth projective curve of genus $2$;
 \item[(3)] if $s=2$, then there exists a genus two curve $C$ and an ample divisor $D\hookrightarrow K$ with a commutative diagram:
  \begin{eqnarray*}
\xymatrix{
\overline{Z_X}\ar@{^{(}->}[r]\ar@{->>}[d] & B\ar[d]^{\rho}_{\mathrm{isogeny}}\\
C \times D\ar@{^{(}->}[r] & JC\times K.}
\end{eqnarray*}
\end{itemize}
Here is a list of possibile $Z_X$ in low dimensions.
\begin{coro}\label{cor6.1}
 \begin{itemize}
  \item[1)] If $n=2$ or $3$, $Z_X$ is always a curve of genus $2$.
  \item[2)] If $n=4$, either $s=1$ or $s=2$ and  then
$Z_X$ is isomorphic to a product  $C_1\times C_2$ of two smooth curves of genus $2$.
   \item[3)] If $n=5$, either $s=1$, or $s=2$ and
$Z_X$ is isomorphic to a product  $C_1\times C_2$ of two smooth curves of genus $2$ or is an \'etale cover of $C\times D$, where $D$ is an ample divisor of an abelian $3$ fold.
 \item[4)]   If $n=6$, either $s=1$, or $s=2$ and
$Z_X$ is isomorphic to a product  $C_1\times C_2$ of two smooth curves of genus $2$ or is an \'etale cover of $C\times D$ as in $3)$, or $s=3$ and $Z_X$ is isomorphic to a product $C_1\times C_2\times C_3$ of three smooth curves of genus $2$.
 \end{itemize}
\end{coro}

We now focus on the case $s=1$.
\begin{lemm}\label{pushforward}
If $s=1$, then we write $Z_X=C$ a smooth curve of genus $2$. We have $f_*(\omega_X)=\omega_C \otimes Q$ for some nontrivial torsion line bundle $Q$ on $C$.
Morevoer,
\begin{enumerate}
\item[1)] we have a decomposition:
\begin{eqnarray}
g_*\omega_X=h^*(\omega_{C}\otimes Q)\bigoplus_t (q_t^*\cQ_t\otimes Q_t),
\end{eqnarray}\label{decomposition1}
where for each $t$,  $q_t: A_X\rightarrow T_t$ is a quotient of abelian varieties with connected fibers, $\cQ_t$ is a M-regular sheaf on $T_t$, $Q_t\notin \PT_t$ is a non-trivial torsion line bundle;
\item[2)]\label{6.2.2}  let $\widetilde{C}\rightarrow C$ be the cyclic \'etale cover induced by $Q$ and let $\wX=X\times_C\widetilde{C}$ be the induced \'etale cover of $X$, then $\wX$ is of maximal Albanese dimension;
\item[3)] let $F$  be a general fiber of $f$, then $F$ is of maximal Albanese fibration and $p_g(F)=1$ hence $q(F)=\dim F$.
\end{enumerate}
 \end{lemm}
\begin{proof}
Note that by the main theorem in \cite{PPS}, $$g_*\omega_X\simeq \bigoplus_t q_t^*\cF_t\otimes Q_t,$$
where each $\cF_t$ is an M-regular coherent sheaf supported on the complex torus $T_t$,
each $q_t: A_X\rightarrow T_t$ is surjective with connected fibers, and each $Q_t$ is a torision line bundle on $A_X$.
Since $h^0(Y, g_*\omega_X)=1$, there exists a unique $t_0$ such that $q_{t_0}^*\cF_{t_0}\otimes Q_{t_0}$ has a non-trivial global section.
Note that the natural morphism $h^*(h_*g_*\omega_X)=h^*(f_*(\omega_X))=h^*(\omega_C \otimes Q)\rightarrow g_*\omega_X$ is injective.
Since $h^0(Y, h^*(\omega_Z\otimes Q))$ is also $1$, this natural injective morphism factors through an injective morphism $$h^*(\omega_C\otimes Q)\rightarrow q_{t_0}^*\cF_{t_0}\otimes Q_{t_0}.$$
Since $h^0(q_{t_0}^*\cF_{t_0}\otimes Q_{t_0})$ is non-zero, the torsion sheaf $Q_{t_0}$ is contained in $V^0(q_{t_0}^*\cF_{t_0}\otimes Q_{t_0})=\PT_{t_0}$ and we may write $q_{t_0}^*\cF_{t_0}\otimes Q_{t_0}=q_{t_0}^*(\cF)$.
Since $V^0(h^*(\omega_C \otimes Q))=\widehat{B}$ is contained in $V^0(q_{t_0}^*\cF_{t_0}\otimes Q_{t_0})=\PT_{t_0}$, the morphism $p: A_X\rightarrow B$ factors through $q_{t_0}: A_X \to T_{t_0}$ and we have the injective morphism on $T_{t_0}$:
$$\varphi: q^*(\omega_C\otimes Q)\rightarrow \cF_{t_0},$$ where $q: T_{t_0}\rightarrow B$ is the natural surjective morphism.
Denote by $\cQ$ the kernel of $\varphi$.
Since $\cF_{t_0}$ is M-regular and $q^*(\omega_X\otimes Q)$ is GV, we conclude that $\cQ$ is also an M-regular coherent sheaf.
On the other hand, $h^0(A_{t_0}, \cQ)=0$. Hence $\cQ=0$ and $\varphi$ is an isomorphism.
Therefore, we may write
\begin{eqnarray}
g_*\omega_X=h^*(\omega_C\otimes Q)\bigoplus_t (q_t^*\cQ_t\otimes Q_t),
\end{eqnarray}\label{decomposition10}
where for each $t$, $Q_t$ is a torsion line bundle on $X$.
Since $h^0(Y, g_*\omega_X)=h^0(Y, h^*(\omega_Z\otimes Q))=1$, none of the $Q_t$'s is contained in $\widehat{T_{t}}$.

 Note that $f_*\omega_X=\omega_C\otimes Q$ is of rank $1$. Hence $p_g(F)=\rank f_*\omega_X=1$.
 On the other hand, let $\pi: \widetilde{C}\rightarrow C$ be the \'etale cover of $C$ induced by the
 torsion line bundle $Q$ and let $\wX$ and $\wY$ be the induced \'etale covers $X\times_C\widetilde{C}$ and
 $Y\times_C\widetilde{C}$. We then consider the fibration
 $\widetilde{f}: \widetilde{X}\xrightarrow{\widetilde{g}}\widetilde{Y}\xrightarrow{\widetilde{h}} \widetilde{C}$. Let $g': \widetilde{X}\rightarrow Y'$ be the Stein facorization of $\widetilde{g}$ and after birational modifications, we may suppose that $Y'$ is smooth.
 By the first part, we know that $\widetilde{h}^*\omega_{\widetilde{C}}$ is a direct summand of
 $\widetilde{g}_*\omega_{\widetilde{X}}$.
  Hence $h^{n-1}(\wY, \widetilde{g}_*\omega_{\wX})>0$. Thus, $h^{n-1}(Y', g'_*\omega_{\wX})>0$.

   By Koll\'ar's splitting,
   \begin{eqnarray*}q(\widetilde{X})=h^{n-1}(\wX, \omega_{\wX})&=&h^{n-1}(Y', g'_*\omega_{\wX})+h^{n-2}(Y', R^1g'_*\omega_{\wX})\\&=&h^{n-1}(Y', g'_*\omega_{\wX})+h^{n-2}(Y',\omega_{Y'})=h^{n-1}(Y', g'_*\omega_{\wX})+q(Y').
   \end{eqnarray*}
 Hence $q(\widetilde{X})>q(Y')$. Since $g': \widetilde{X}\rightarrow Y'$ is a fibration, $Y'$ is of maximal Albanese dimension and $\dim Y'=\dim X-1$, we conclude that
  $\wX$ is of maximal Albanese dimension and hence so is $F$.

   Since $F$ is of maximal Albanese dimension and $p_g(F)=1$, we know that $h^i(F, \cO_F)=h^0(F, \Omega_F^i)=\binom{\dim F}{i}$ (see for instance \cite[Proposition 6.1]{CDJ}).
 \end{proof}

 \begin{theo}\label{kodaira-dimension}
Assume that $X$ is a fake torus of dimesnion $n\geq 3$ with $\dim Y=n-1$. Then
\begin{itemize}
\item[(1)] let $\wX=X\times_C\widetilde{C}$ defined as in $2)$ of Lemma \ref{pushforward},
then $a_{\wX}$ is a finite morphism onto its image;
\item[(2)] $X$ is not of general type.
 \end{itemize}
  \end{theo}

 \begin{proof}

 Note that $Z=C$ is a smooth curve of genus $2$ and $Y=C\times_{JC}A_X$ and $F$ is  of maximal Albanese dimension with $p_g(F)=1$. We know from Lemma \ref{pushforward} that $q(\wX)\geq q(\wY)+1=q(\widetilde{C})+n-1$. Moreover, F is also a general fiber of  $\widetilde{X}\rightarrow \widetilde{C}$. By Lemma \ref{pushforward}, $q(F)=n-1$. Hence $q(\widetilde{X})-q(\widetilde{C})\leq q(F)=n-1$. Thus, $q(\wX)=q(\widetilde{C})+n-1$.

We now consider the induced fibration $\wf: \wX\rightarrow \widetilde{C}$  in Lemma \ref{pushforward} and the following commutative diagram:
\begin{eqnarray*}
\xymatrix{
\wX\ar[d] \ar[r]^{G} & X\ar[d]^g\\
\widetilde{Y}\ar[d]\ar[r]^G{} & Y\ar[d]^h\\
\widetilde{C}\ar[r]^{G} & C.}
\end{eqnarray*}
Let $G$ be the Galois group of the cover $\widetilde{C} \to C$.
We  define $\widetilde{M}:=\widetilde{C}\times_{J\widetilde{C}}A_{\widetilde{X}}$ and let $K$ be the neutral component of the kernel of $A_{\wX}\rightarrow J\widetilde{C}$.
We then have a natural generically finite morphism $\wX\rightarrow \widetilde{M}$ and a surjective morphism $\widetilde{M}\rightarrow \wY$.
Note that these morphisms are $G$-equivariant.  Let $M=\widetilde{M}/G$.
We then have the induced morphisms on the quotient: $g: X\xrightarrow{\rho} M\xrightarrow{\phi} Y$.

We claim that $h^2(M, \mathbb{Q})=h^2(X, \mathbb{Q})$.

Note that $H^2(M, \mathbb{Q})=H^2(\widetilde{M}, \mathbb{Q})^G$ and
$$H^2(\widetilde{M}, \mathbb{Q})=H^2(\widetilde{C}, \mathbb{Q})\oplus \big(H^1(\widetilde{C}, \mathbb{Q}))\wedge H^1(K, \mathbb{Q})\big)\oplus H^2(K, \mathbb{Q}).$$
Let $K'$ be the neutral component of the kernel of $A_X\rightarrow JC$, which is also a fiber of $h$.
Then we have the quotient morphism $K\rightarrow K'$ by $G$. 
Hence $H^1(K, \mathbb{Q})^G\simeq H^1(K', \mathbb{Q})$ and there exists only one non-trivial character $\chi$ of $G$ such that $H^1(K, \mathbb{Q})^{\chi}\neq 0$ and hence $\dim H^1(K, \mathbb{Q})^{\chi}=2$.

Thus
\begin{eqnarray*}
H^2(\widetilde{M}, \mathbb{C})^G&=& H^2(\widetilde{C}, \mathbb{C})^G\oplus \big(H^1(\widetilde{C}, \mathbb{C}))^G\wedge
H^1(K, \mathbb{C})^G\big)\oplus \big(H^1(\widetilde{C}, \mathbb{C}))^{\chi^*}\wedge H^1(K, \mathbb{C})^{\chi}\big)\\&&\oplus H^2(K, \mathbb{C})^G\\
&=& H^2(C, \mathbb{C})\oplus \big(H^1(C, \mathbb{C}))\wedge H^1(K', \mathbb{C})\big)\oplus \big(H^1(\widetilde{C}, \mathbb{C}))^{\chi^*}\wedge H^1(K, \mathbb{C})^{\chi}\big)\\&&\oplus H^2(K, \mathbb{C})^G.
\end{eqnarray*}
We also have
$$H^2(Y, \mathbb{Q})=H^2(C, \mathbb{Q})\oplus \big(H^1(C, \mathbb{Q}))\wedge H^1(K', \mathbb{Q})\big)\oplus H^2(K', \mathbb{Q}).$$
It is easy to see that $h^2(K, \mathbb{Q})^G=\dim (\wedge^2H^1(K, \mathbb{Q}))^G=h^2(K', \mathbb{Q})+1$ and,  for any non-trivial character $\psi$, $\dim H^1(\widetilde{C}, \mathbb{Q})^{\psi}=2$. Hence $h^2(M, \mathbb{Q})=\dim H^2(\widetilde{M}, \mathbb{Q})^G=h^2(Y, \mathbb{Q})+1+4=h^2(X, \mathbb{Q})$.

Since $h^2(X, \mathbb{Q})=h^2(M, \mathbb{Q})$ and both $X$ and $M$ are smooth projective varieties, the surjective morphism $\rho: X\rightarrow M$ is finite. Then so is the induced morphism on the \'etale covers $\wX\rightarrow \widetilde{M}$. Hence $a_{\wX}$ is finite onto its image. However, $\chi_{top}(\wX)=\chi_{top}(X)=0$. By \cite[Theorem 1 of the appendix]{DJL}, $\wX$ can not be of general type and neither can $X$.
  \end{proof}

\begin{prop}\label{surface}
Let $X$ be a fake torus of dimension $2$.
Then $X$ is a minimal projective surface with $\kappa(X)=1$.
Furthermore, there exists a finite abelian group $G$ acting faithfully on an elliptic curve $E$
and on a smooth projective curve $D$ of genus $\geq 3$ such that $E/G\simeq \mathbb{P}^1$, $D/G=C$ is a smooth curve of  genus $2$,
and $X$ is isomorphic to the diagonal quotient $(D\times E)/G$.
\end{prop}
\begin{proof}
 Let $\wX\rightarrow X$ and $\widetilde{C}\rightarrow C$ be the \'etale covers induced by $Q$ as in 2) of Lemma \ref{pushforward}.
Since $a_{\wX}$ is finite and $X$ is not of general type,  we conclude that a general fiber of $\wX\rightarrow \widetilde{C}$ is isogenous to the kernel $A_{\wX}\rightarrow J\widetilde{C}$. Hence $f: X\rightarrow C$ is isotrivial with a smooth fiber isomorphic to an elliptic curve $E$.
Moreover, by Kawamata's theorem (\cite[Theorem 15]{kaw}),  we know that a fiber of $\wX\rightarrow \widetilde{C}$ is either smooth or is a multiple of a smooth curve.
Hence both $\wX\rightarrow \widetilde{C}$
and $f: X\rightarrow C$ are quasi-bundles in the terminology of \cite{ser}.

By the main result of \cite{ser}, we conclude that there exists a Galois cover $D\rightarrow C$ with Galois group $G$ such that $D\times_CX\simeq D\times E$.
Moreover, $X\simeq (D\times E)/G$, where $G$ acts faithfully on both factors and the action on the product is the diagonal action.
Since $h^1(X, \cO_X)=2=h^1(C, \cO_C)$, we conclude that $E/G\simeq \mathbb{P}^1$.

On the other hand, any smooth surface isomorphic to  $(D\times E)/G$ with $D/G$ a smooth projective curve of genus $2$ and $E/G\simeq \mathbb{P}^1$ is a fake torus of dimension $2$.
\end{proof}

A fake torus of dimension $3$ has Kodaira dimension $1$ or $2$.
With some efforts, in both cases, we can prove a similar structural result as in the surface case.
Here is a typical example.
\begin{exam}Let $G$ be an abelian group acting faithfully on an elliptic curve $E$ by translation. Let $S$ be a smooth projective surface such that $G$ acts faithfully on $S$ and $S/G$ is a fake torus of dimension $2$. Then the diagonal quotient $(S\times E)/G$ is a fake torus in dimension $3$.
\end{exam}

When $X$ is a fake torus of dimension $4$ and $Y=Z=C_1\times C_2$ is a product of two smooth curves of genus $2$.  We know that $f_*\omega_X=\omega_Z\otimes Q$. Hence $f$ is a fibration and a general fiber $F$ of $f$ has $p_g(F)=1$. Moreover, we can verify by Koll\'ar's splitting and the Hodge diamond of $X$ that $h^0(Z, R^1f_*\omega_X)=4$ and $h^1(Z, R^1f_*\omega_X)=2$. Hence $F$ is an irregular surface. We do not know whether or not $F$ is an abelian surface. In general, for a fake torus $X$, we do not have a systematic way to study the fibers of $a_X$.

Finally we conclude by the following questions:
\begin{qu}
\item[(1)] Does there exist  fake tori of general type ?
\item[(2)] Does there exist a fake torus $X$ such that $Z_X$ is not a product of genus $2$ curves ?
\end{qu}

\end{document}